%% file: Tmanifolds.tex
\documentclass{amsart}

\usepackage[utf8]{inputenc} 
\usepackage[T1]{fontenc}    
\usepackage{hyperref}       
\usepackage{url}            
\usepackage{booktabs}       
\usepackage{amsfonts}       
\usepackage{amsmath}
\usepackage{amsthm}
\usepackage{nicefrac}       
\usepackage{microtype}      
\usepackage{graphicx}
\usepackage{natbib}
\usepackage{doi}
\usepackage{xcolor}
\usepackage{amssymb}

\usepackage{leftindex}

\usepackage{lmodern}

\usepackage{mathrsfs}
\usepackage{pgfplots}

\usepackage{caption}
\usepackage{subcaption}

\pgfplotsset{compat=1.18}

\theoremstyle{plain}
\newtheorem{theo}{Theorem }[section]
\newtheorem{lemma}[theo]{Lemma }
\newtheorem{prop}[theo]{Proposition }
\newtheorem{coro}[theo]{Corollary }

\theoremstyle{definition}
\newtheorem{defi}[theo]{Definition }
\newtheorem{Rem}[theo]{Remark }
\newtheorem{Ex}[theo]{Example }

\newcommand{\Hom}{\operatorname{Hom}}
\newcommand{\Vol}{\operatorname{Vol}}
\newcommand{\Bary}{\operatorname{Bar}}
\newcommand{\Conv}{\operatorname{Conv}}
\newcommand{\spn}{\operatorname{span}}
\newcommand{\Cone}{\operatorname{Cone}}

\title{On the topology of T-manifolds of higher codimension}
\author{Enzo Pasquereau}
\date{\today}

\begin{document}

\begin{abstract}
	This paper undertakes the study of the topology of T-manifolds of arbitrary codimension obtained by combinatorial patchworking with real phase structure as described in \cite{brugalle2024combinatorial}.
    We prove new bounds on the number of connected components of T-curves and T-surfaces.
    For sufficiently high codimension, this improves the results from \cite{brugalle2024combinatorial}.
    
    In addition, we present a new description of patchworking à la Viro for T-manifold of codimension 2.
    We use this method to construct a family of maximal real algebraic curves in $\mathbb RP^3$.
\end{abstract}

\maketitle
\tableofcontents

\section{Introduction}

Viro's patchworking is a powerful method for constructing real algebraic varieties with prescribed topology\cite{Viro84, viro2006patchworkingrealalgebraicvarieties, risler1992construction}.
In its original form, one starts with a collection $(f_i)_i$ of non-degenerate polynomials whose Newton polytope $\Delta_i$ form a convex (also known as regular or coherent) subdivision $\Gamma$ of a polytope $\Delta$.
The polynomials $f_i$ are chosen so that the coefficients of common monomials coincide.
From this data, Viro's construction produces a real algebraic hypersurface with Newton polytope $\Delta$, whose topology is described as a gluing of the real hypersurfaces defined by the polynomials $f_i$.

This procedure has proved to be extremely effective in the study of topology of real algebraic varieties.
For instance, Viro used it to describe all isotopy types of smooth planar curves of degree 7 \cite{Viro84}.
It has also inspired numerous developments and generalizations.

One of these is combinatorial patchworking.
It corresponds to the case where the polytopes $\Delta_i$ are simplices (i.e. $\Gamma$ is a triangulation) and where the monomials of each polynomial $f_i$ correspond exactly to the vertices of $\Delta_i$ \cite{contreRagsdale,Logpaper}.
As the name suggests, this method reduces to a combinatorial description, given by a sign distribution on the vertices of the triangulation.
The piecewise linear manifold obtained by this procedure is called a T-hypersurface.
Here, the assumption of convexity of the triangulation is unnecessary for constructing T-hypersurfaces from the combinatorial description.
It remains, however, a key hypothesis to apply Viro's method and deduce an algebraic hypersurface with corresponding topology.

This method has yielded significant advances (including the disproving of Ragsdale conjecture \cite{Ragsdale} and the construction of maximal hypersurfaces \cite{ItenbergMaxTcurve} \cite{itenberg1997topology} \cite{Alois}).
An extension of patchworking to the case of complete intersections has been developed by Sturmfels (for combinatorial patchworking) \cite{sturmfels1994viro} and then, by Bihan (for general patchworking) \cite{Bihan}.

Recently, Renaudineau and Shaw \cite{RS18} proved a conjecture by Itenberg giving an upper bound on the Betti number $b_p(X) := \dim H_p(X; \mathbb Z_2)$ of a primitive T-hypersurface $X$.
A triangulation is said to be primitive (or unimodular) when every full-dimensional simplex $\sigma$ of the triangulation has a lattice volume $\Vol(\sigma)$ of 1.
A central concept introduced by Renaudineau and Shaw for their proofs is the one of real phase structure, which reformulates sign distribution.

This result has, since, been generalized in two directions.
Firstly, real phase structure have been extended to not only describe patchworks of hypersurfaces but also patchworks of more general varieties \cite{rau2023realphase,rau2022real,brugalle2024combinatorial}.

Another generalization arose from forgetting the convexity of the triangulation.
Indeed, combinatorial patchworking has been reinterpreted in the framework of tropical geometry, where primitive convex triangulations become non-singular tropical hypersurfaces.
This geometric point of view provides new intuition and additional tools notably tropical homology. 
However, going back to triangulation, Brugallé, López de Medrano, and Rau \cite{brugalle2024combinatorial} proved that the inequalities still hold without any convexity assumption.

\begin{theo}[\cite{brugalle2024combinatorial}]
    Let $X$ be a T-manifold of codimension $k$ with a smooth Newton polytope $\Delta$. Then,
    \begin{equation}\label{boundhodge}
        \forall p \geq 0, ~~~   b_p(X) \leq \sum_{q=0}^{n-k} h^{p,q}(Y_\Delta^k)
    \end{equation}
    where $Y^k_{\Delta}$ denotes a smooth complete intersection of $k$ hypersurfaces with Newton polytope $\Delta$ in the toric variety $\mathbb CX_{\Delta}$ associated to $\Delta$ and $h^{p,q}(Y^{k}_{\Delta})$ is it $(p,q)$-Hodge number.
\end{theo}

We refer to Section \ref{sec:construction} for precise definitions of T-manifolds, real phase structures and smooth polytopes.

In the setting of Renaudineau and Shaw, a T-hypersurface is homeomorphic to the real part $\mathbb RX$ of a smooth algebraic hypersurface $\mathbb CX$ with Newton polytope $\Delta$.
In this case, the sum of all the Betti numbers satisfies the Smith-Thom inequality:
\[ \sum_{p \geq 0} b_p(\mathbb RX) \leq \sum_{p \geq 0} b_p(\mathbb CX)\]
A real algebraic variety attaining this bound is said to be maximal.
Without convexity and in higher codimension, no homeomorphism is known in general between T-manifold and any algebraic variety.
Still, by analogy with real algebraic geometry, we adopt the following terminology.
\begin{defi}{\ }

A T-manifold for which the inequalities (\ref{boundhodge}) are equalities for each $p$ is called \textbf{maximal}.

Let $\Delta$ be a polytope and $k$ an integer. For any $d \geq 1$, let $X_d$ be a T-manifold of codimension $k$  with  Newton polytope the dilation $d\Delta$.
    We say the family $(X_d)_{d \geq 1}$ is \textbf{asymptotically maximal} on $\Delta$ if
\[\forall p \geq 0,  ~~~ b_p(X_{d}) \sim_{d \to \infty} \sum_q h^{p,q}(Y^k_{d \Delta}).\]
    
\end{defi}

T-hypersurfaces have been studied for almost thirty years, and many examples of maximal T-hypersurfaces (\cite{ItenbergMaxTcurve, itenberg1997topology, Alois})
and asymptotically maximal families of T-hypersurfaces (\cite{AsympmaxViro, BertrandThese}) have been constructed. 
In contrast, since their definition is more recent, not much is known for T-manifolds of higher codimension.

The present paper explores this new setting.
We are interested in "extremal" T-manifolds with large Betti numbers, and in particular, on the number of connected components.

We establish two new upper bounds on the number of connected components for T-curves and T-surfaces.
\begin{theo}Let $\Delta$ be a smooth lattice polytope of dimension $n$.
    \begin{itemize}
        \item A T-curve $X$ with Newton polytope $\Delta$ satisfies:
            \[ b_0(X) \leq \frac{n+1}{3}\Vol(\Delta) \]
        \item A T-surface $X$ with Newton polytope $\Delta$ satisfies:
            \[ b_0(X) \leq \frac{7n^2+5n+12}{30}\Vol(\Delta)\]
    \end{itemize}
\end{theo}

We also compute the following asymptotic of the right-hand side of \eqref{boundhodge}, when $p=0$:

 \begin{lemma}
     The sum $\sum_{q=0}^{n-k}h^{0,q}(Y_{d\Delta}^k)$ is a polynomial in $d$ whose leading term is
     \[\frac{k!}{n!} {n \brace k} \Vol(\Delta)d^n,\]
     where  ${n \brace k}$ denotes the Stirling number of the second kind.
 \end{lemma}

Comparing these two results yields the following result on the asymptotic maximality of T-manifolds:
\begin{coro}
    Let $\Delta$ be a smooth lattice polytope of dimension $n$.
    \begin{itemize}
        \item If $n \geq 6$, there is no asymptotically maximal family of T-curves on $\Delta$.
        \item If $n \geq 65$, there is no asymptotically maximal family of T-surfaces on $\Delta$.
    \end{itemize}
\end{coro}

For T-curves, an additional obstruction to maximality arises from the planarity of the dual graph of the triangulation. This yields even stronger restrictions:

\begin{theo}
    Let $\Delta$ be a smooth lattice polytope of dimension $n$.
    If $n \geq 4$ and $d \geq n+1$, there is no maximal T-curve with Newton polytope $d\Delta$.
\end{theo}

Real phase structures, and thus T-manifolds, are more abstruse to describe than sign distributions and T-hypersurfaces.
Even the existence of real phase structure on a given unimodular triangulation is not immediately clear.

In a second part of this text, we introduce a new method for constructing real phase structures.
Given an orientation of the edges of a triangulation and two real phase structures, we describe a new real phase structure.
The associated T-manifold can be interpreted as a combinatorial equivalent of real stable intersections in tropical geometry (see \cite[Section 4.3]{TropicalBook} for background on stable intersections).

This construction provides:

\begin{theo}
    Let $\Gamma$ be a triangulation of $\Delta$.
    For any $k$ between 0 and $n$, there exists a $k$-real phase structure on $\Gamma$.
\end{theo}

This combinatorial real stable intersection is particularly well-suited for codimension 2, since the description of real phase structures for T-hypersurfaces reduces to sign distributions.
This yields a combinatorial description of a subclass of T-manifolds of codimension 2 reminiscent of classical combinatorial patchworking.

Combined with Sturmfels' method, we show a realizability criterion ensuring that the resulting T-manifolds have the topology of smooth complete intersections. 
\begin{theo}
    Let $\Delta$ be a smooth lattice polytope, $\Gamma$ be a convex triangulation of $\Delta$, $\mathcal O$ an acyclic orientation of $\Gamma$ and two sign distributions $\mu_1, \mu_2$.
    Consider the T-hypersurface $X_{\mathcal E_{\mu_1}}$ and the T-manifold $X_{\mathcal E_{\mu_1} \cap_{\mathcal O} \mathcal E_{\mu_2}}$ of codimension 2.
    There exists two transverse real algebraic hypersurfaces $H_1$ and $H_2$ with Newton polytope $\Delta$ in $\mathbb RX_\Delta$ such that
    the triple $(\widetilde{\Delta}, X_{\mathcal E_{\mu_1}}, X_{\mathcal E_{\mu_1} \cap_{\mathcal O} \mathcal E_{\mu_2}})$ is homeomorphic to $(\mathbb RX_\Delta, H_1, H_1 \cap H_2)$.
\end{theo}

Using this description we explicitly construct maximal T-curves in $\mathbb RP^3$.
\begin{theo}
    For each integer $d\geq1$, there exists a pair $(\Sigma_d, C_d)$ where $\Sigma_d$ is a maximal T-surface and $C_d$ a maximal T-curve both with Newton polytope $d\Delta_3$ such that $C_d \subset \Sigma_d$.
    The triple $(\widetilde{d\Delta_3}, \Sigma_d, C_d)$ is homeomorphic to $(\mathbb RP^3, H_1, H_1 \cap H_2)$ where $H_1$ and $H_2$ are transverse maximal real algebraic hypersurfaces of degree $d$.
\end{theo}
Here $d\Delta_3$ is the tetrahedron with vertices $(0,0,0)$, $(d,0,0)$, $(0,d,0)$ and $(0,0,d)$.

Note that the pair $(H_1, H_1 \cap H_2)$ is not maximal.

\subsubsection*{Organization of the paper}

\begin{itemize}
    \item In Section 2, we present the construction of T-manifolds via real phase structures as presented in \cite{brugalle2024combinatorial}.
    \item Section 3 is devoted to a local study of real phase structures, which provides the key technical ingredients needed for our later results.
    \item In Section 4, we establish our results on the number of connected components of T-curves and T-surfaces.
    \item Section 5 introduces a construction of real phase structures and its application to T-manifolds of codimension 2.
    \item In Section 6, we apply this method to construct a family of maximal T-curves in the three-dimensional projective space.
\end{itemize}

\subsubsection*{Acknowledgments}
We are deeply grateful to Erwan Brugallé for his guidance, for many insightful discussions, and for his helpful suggestions and advice during the preparation of this work.

\section{Patchworking of T-manifolds}\label{sec:construction}

\subsection{Real phase structure and T-manifold}

We begin by defining T-manifold following \cite{brugalle2024combinatorial}.
For T-hypersurfaces, this construction has first been described by Renaudineau and Shaw in the context of tropical geometry \cite{RS18}.

Let $\Delta \subset \mathbb R^n$ be a full-dimensional convex polytope with vertices in the lattice $\mathbb Z^n$. Such polytopes are referred to as lattice polytopes.
Consider a triangulation $\Gamma$ of $\Delta$ with vertices in $\mathbb Z^n$. 
We denote by $F_k(\Gamma)$ the set of $k$-simplices of $\Gamma$, and for a simplex $\tau$, we denote by $F_k(\tau)$ the set of its $k$-faces.
A simplex with vertices $v_0, v_1, \dots v_k$ will be denoted by $[v_0, v_1, \dots, v_k]$.
We suppose throughout the text that $\Gamma$ is unimodular meaning that all simplices of $F_n(\Gamma)$ have euclidean volume $\frac{1}{n!}$.
The lattice volume of a full-dimensional lattice polytope $\Delta \subset \mathbb R^n$, denoted by $\Vol(\Delta)$, is the euclidean volume of $\Delta$ divided by $\frac{1}{n!}$ the volume of a unimodular simplex.

For each simplex $\sigma$, we define its integral tangent space $T(\sigma) \subset \mathbb Z^n$ as the sub-lattice spanned by the differences of two points of $\sigma \cap \mathbb Z^n$.
We denote by $T_2(\sigma) := T(\sigma) \otimes \mathbb F_2 \subset \mathbb F_2^n$ and $T_2^\perp (\sigma) \subset (\mathbb F_2^n)^{\vee} = \Hom(\mathbb F_2^n, \mathbb F_2)$ the orthogonal of $T_2(\sigma)$ in the dual of $\mathbb F_2^n$.
We suppose throughout the text that $\Delta$ is smooth, that is for any vertex $v$ of $\Delta$, there is a decomposition $\mathbb Z^n = \bigoplus T(\rho)$, where the sum runs over all the edges $\rho$ of $\Delta$ incident to $v$.

Fix an affine space $\mathcal Q^n$ directed by $(\mathbb F_2^n)^\vee$.
We will identify $\mathcal Q^n$ with the set $\{+,-\}^n$ of orthant of $\mathbb R^n$,
so the action of $(\mathbb F_2^n)^{\vee}$ on $\mathcal Q^n$ corresponds to the one of the reflection group generated by reflections across the coordinate hyperplanes.

\begin{defi}
    A real phase structure $\mathcal{E}$ on the $k$-skeleton of $\Gamma$ is a choice for each $\sigma \in F_k(\Gamma)$ of an affine subspace $\mathcal E(\sigma)$ of $\mathcal Q^n$ directed by $T_2^\perp(\sigma)$ satisfying the following \textbf{parity condition}:
    for each $\tau \in F_{k+1}(\Gamma)$ and  $s \in \mathcal Q^n$ there is an even number of $\sigma \in F_k(\tau)$ such that $s \in \mathcal E(\sigma)$.
\end{defi}

We also call $\mathcal E$ a $k$-real phase structure on $\Gamma$.

\begin{Ex}\label{ex:E_1}
    Consider the triangle $\Delta_2 = [a, b, c]$ with vertices $a=(0,0), b=(1,0), c=(0,1)$. We endowed it with its trivial (and unique triangulation).
    A real phase structure $\mathcal E_1$ on the 1-skeleton of $\Delta_2$ is given by
    \begin{itemize}
        \item $\mathcal E_1([a,b]) =\{(+,+), (+,-) \}$
        \item $\mathcal E_1([a,c]) =\{(+,+), (-,+) \}$
        \item $\mathcal E_1([b,c]) =\{(+,-), (-,+) \}$
    \end{itemize}
\end{Ex}

In the following, given a simplex $\tau$ of $\Gamma$, we denote by $\mathcal E(\tau) = \bigcup_{\sigma \in F_k(\tau)} \mathcal E(\sigma)$.

Given a real phase structures on the $k$-skeleton of $\Gamma$, we define a PL manifolds called T-manifolds by the following construction.

For each $s \in \mathcal Q^n$, consider a copy $s(\Delta)$ of $\Delta$ with the triangulation $s(\Gamma)$, which is a copy of $\Gamma$.
We say that $s(\Delta)$ is the symmetric copy of $\Delta$ in the orthant $s$.

We form a topological space gluing together the $s(\Delta)$ on their faces.
For each face $F$ of $\Delta$, we identify the face $s(F)$ and $t(F)$ each time that $t-s \in T_2^\perp(F)$.
The resulting space is denoted by $\widetilde{\Delta}$, and it is triangulated by the extended triangulation $\widetilde \Gamma$ obtained by the gluing of the copies $s(\Gamma)$.

Let $\Bary(\widetilde \Gamma)$ be the barycentric subdivision of $\widetilde \Gamma$. The vertices $v_\sigma$ of $\Bary(\widetilde{\Gamma})$ are the barycenter of the simplex $\sigma$ of $\widetilde{\Gamma}$.
For a simplex $\tau \in \Gamma$ and an orthant $s \in \mathcal{Q}^n$,
we associate to the simplex $s(\tau)$ of $\widetilde{\Gamma}$ the subspace $C_{\mathcal E}(s(\tau))$ defined as the full subcomplex of $\Bary(\widetilde \Gamma)$ whose vertices are the barycenters $v_{s(\sigma)}$ of the faces of $s(\sigma)$ of $s(\tau)$ such that $s \in \mathcal E(\sigma)$.
Each subcomplex $C_{\mathcal E}(s(\tau))$ can be made of several simplices of $\widetilde{\Gamma}$ or empty if $s \notin \mathcal E(\tau)$.

\begin{Ex}
    This construction is illustrated in the simple case of Example \ref{ex:E_1} in Figure \ref{n2k1}.
    On the top left, the four symmetric copies of $\Delta_2$ appear in their corresponding quadrant.
    On the top right, we obtain $\widetilde{\Delta_2}$ by gluing these four copies together. The boundary edges must be identified according to the arrows.
    At the bottom, the barycentric subdivision of $\widetilde{\Delta_2}$ is shown (the identification of the boundary edges is not shown).
    For the triangle $\tau =\Delta_2$, if $s \in \mathcal E_1(\Delta_2)$,  the corresponding subcomplex $C_{\mathcal E_1}(s(\tau))$ is drawn in green.
    It is made of two edges and three vertices of the barycentric subdivision $\Bary(s(\Delta_2))$.
    For an edge $\gamma$ of $\Delta_2$, if $s \in \mathcal E_1(\gamma)$, the corresponding 0-cell $C_{\mathcal E_1}(s(\gamma))$ is circled in a darker green.
\end{Ex}
    
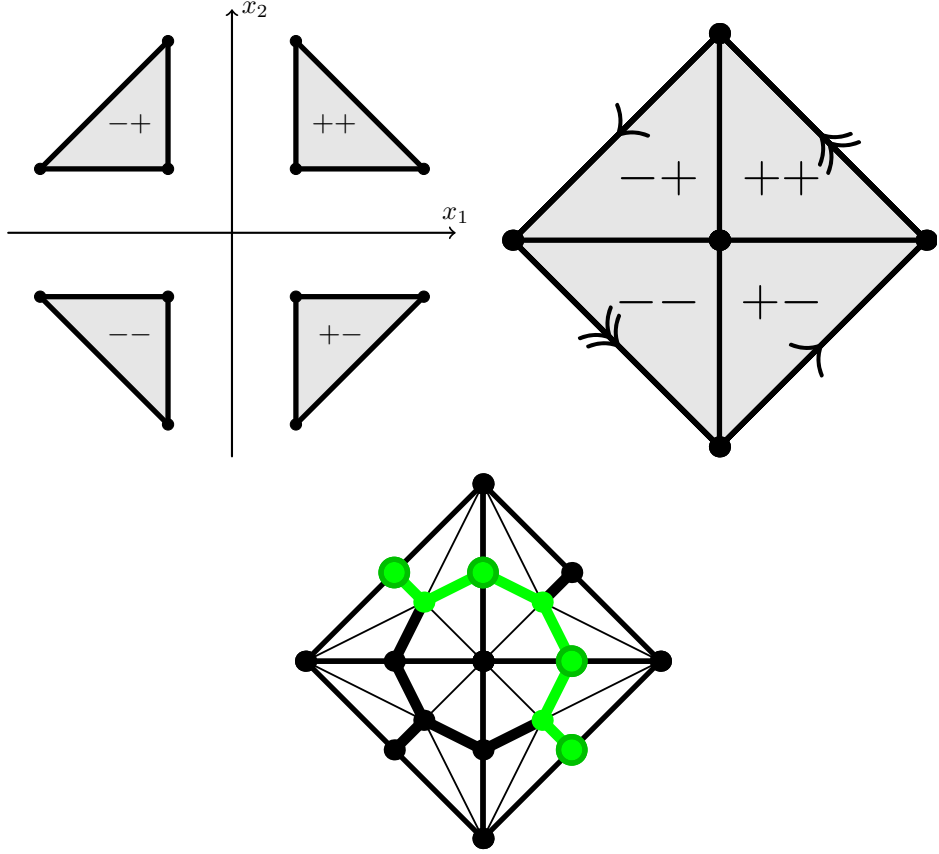
\begin{figure}[h]
    \centering
    \begin{subfigure}{0.48\textwidth}

 \begin{tikzpicture} [scale=0.85, line cap=round,line join=round,x=1cm,y=1cm]

\draw[->, thick] (-3.5,0)--(3.5,0) node[above]{$x_1$};
\draw[->, thick] (0,-3.5)--(0,3.5) node[right]{$x_2$};

\fill[line width=2pt,fill=black,fill opacity=0.1] (1,1) -- (1,3) -- (3,1) -- cycle;
\fill[line width=2pt,fill=black,fill opacity=0.1] (-1,3) -- (-1,1) -- (-3,1) -- cycle;
\fill[line width=2pt,fill=black,fill opacity=0.1] (1,-1) -- (1,-3) -- (3,-1) -- cycle;
\fill[line width=2pt,fill=black,fill opacity=0.1] (-3,-1) -- (-1,-1) -- (-1,-3) -- cycle;
\draw [line width=2pt] (1,1)-- (1,3);
\draw [line width=2pt] (1,3)-- (3,1);
\draw [line width=2pt] (3,1)-- (1,1);
\draw [line width=2pt] (-1,3)-- (-1,1);
\draw [line width=2pt] (-1,1)-- (-3,1);
\draw [line width=2pt] (-3,1)-- (-1,3);
\draw [line width=2pt] (1,-1)-- (1,-3);
\draw [line width=2pt] (1,-3)-- (3,-1);
\draw [line width=2pt] (3,-1)-- (1,-1);
\draw [line width=2pt] (-3,-1)-- (-1,-1);
\draw [line width=2pt] (-1,-1)-- (-1,-3);
\draw [line width=2pt] (-1,-3)-- (-3,-1);
\begin{scriptsize}
\draw [fill=black] (1,1) circle (2.5pt);
\draw [fill=black] (1,3) circle (2.5pt);
\draw [fill=black] (3,1) circle (2.5pt);
\draw[color=black] (1.6,1.7) node {\large $++$};
\draw [fill=black] (-1,3) circle (2.5pt);
\draw [fill=black] (-1,1) circle (2.5pt);
\draw [fill=black] (-3,1) circle (2.5pt);
\draw[color=black] (-1.6,1.7) node {\large $-+$};
\draw [fill=black] (1,-1) circle (2.5pt);
\draw [fill=black] (1,-3) circle (2.5pt);
\draw [fill=black] (3,-1) circle (2.5pt);
\draw[color=black] (1.7,-1.6) node {\large $+-$};
\draw [fill=black] (-3,-1) circle (2.5pt);
\draw [fill=black] (-1,-1) circle (2.5pt);
\draw [fill=black] (-1,-3) circle (2.5pt);
\draw[color=black] (-1.6,-1.6) node {\large $--$};
\end{scriptsize}

 \end{tikzpicture}   
   \end{subfigure}
\hfill
\begin{subfigure}{0.48\textwidth}
    
\includegraphics{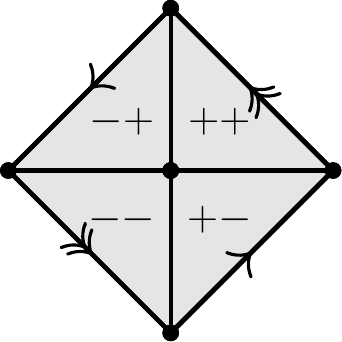} 

\end{subfigure}

\vspace{0.5em}

\begin{subfigure}{0.8\textwidth}
    \centering

\includegraphics{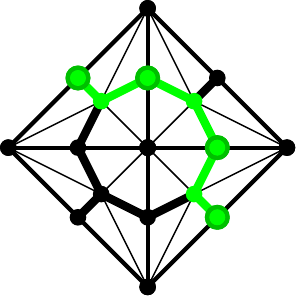}

\end{subfigure}
\caption{The construction of $\widetilde{\Delta_2}$ the T-curve associated to $\mathcal E_1$}
\label{n2k1}
\end{figure}

\begin{defi}
    A regular cell complex subdividing a topological space $X$ is a finite set $\mathcal C$ of topological closed balls (called cells) in $X$ satisfying:
    \begin{itemize}
        \item the interiors $ \mathring \sigma$ of the cells $\sigma \in \mathcal C$ partition $X$
        \item the boundary $\partial \sigma$ of a cell $\sigma \in \mathcal C$ is a union of cells of $\mathcal C$
    \end{itemize}
    The underlying space $|\mathcal C|$ of a cell complex $\mathcal C$ is the topological space $X$.
\end{defi}

\begin{prop}[Proposition 4.11, \cite{brugalle2024combinatorial}]\label{PLmanifold}
The collection of all the of the $C_{\mathcal E}(\tau)$, for $\tau \in \widetilde \Gamma$, defines a regular cell complex.
The $m$-cells of this complex are precisely the $C_{\mathcal E}(s(\tau))$ where $\tau$ is a $(m+k)$-simplex of $\Gamma$ and $s \in \mathcal E(\tau)$.
The underlying space $X_\mathcal E$ of this cell decomposition is a PL manifold of dimension $n-k$
and the cells $C_\mathcal E(\sigma)$ are PL balls.
\end{prop}

The space $X_\mathcal E$ is the T-manifold associated to the real phase structure $\mathcal E$,
and we refer to the above decomposition as the canonical cell decomposition of $X_\mathcal E$.
By construction, we have $C_{\mathcal E}(s(\tau)) = X_\mathcal E \cap s(\tau)$, for any simplex $\tau$ of $\Gamma$ and orthant $s \in \mathcal Q^n$ such that $s \in \mathcal E(\tau)$.

\begin{defi}
    Let $\Delta$ be a lattice polytope.
    A T-manifold of codimension $k$ with  Newton polytope $\Delta$ is a T-manifold $X_{\mathcal E}$,
    where $\mathcal E$ is a real phase structure on the $k$-skeleton of a triangulation of $\Delta$. 
\end{defi}

\begin{Ex}\label{ex1}
The T-curve $X_{\mathcal E_1}$ with  Newton polytope $\Delta_2$, from Example \ref{ex:E_1}, is represented in Figure \ref{n2k1}.
\end{Ex}

\begin{Ex}\label{ex0}
For any 0-simplex (i.e. vertex) $v$, we have $T_2^\perp(v) = (\mathbb F_2 ^n)^\vee$.
Thus, for any lattice polytope $\Delta$ and triangulation $\Gamma$, there is always a unique real phase structure $\mathcal E_0$ on the 0-skeleton of $\Gamma$.
It is defined by $\mathcal E_0(v)= \mathcal Q^n$.
The parity condition is well satisfied since for any edge $[v,v']$ of $\Gamma$, we have $\mathcal E_0(v)=\mathcal E_0(v')$.
It follows, that independently of $\Gamma$, the T-manifold $X_{\mathcal{E}_0}$ is always the full space $\widetilde{\Delta}$.
Furthermore, the canonical cell decomposition is exactly the subdivision $\widetilde{\Gamma}$.
\end{Ex}

\begin{Ex}\label{ex2}
    For $\Delta_3 = [a, b, c, d]$ with vertices $a=(0,0,0)$, $b=(1,0,0)$, $c=(0,1,0)$ and $d=(0,0,1)$,
    a real phase structure $\mathcal E_2$ on the 1-skeleton of  $\Delta_3$ is given by
    \begin{itemize}
        \item $\mathcal E_2([a,b]) =\{(+,+,+), (+,-,+), (+,+,-),(+,-,-) \}$
        \item $\mathcal E_2([a,c]) =\{(+,-,+), (-,-,+),(+,-,-),(-,-,-) \}$
        \item $\mathcal E_2([a,d]) = \{(+,+,+), (-,+,+), (+,-,+),(-,-,+) \}$
        \item $\mathcal E_2([b,c]) = \{(+,+,+), (+,+,-),(-,-,+), (-,-,-) \}$
        \item $\mathcal E_2([b,d]) = \{(-,+,+), (-,-,+),(+,+,-),(+,-,-) \}$
        \item $\mathcal E_2([c,d]) = \{(+,+,+),(-,+,+), (+,-,-), (-,-,-) \}$
    \end{itemize}
    The T-surface $X_{\mathcal E_2}$ is represented in Figure \ref{n3k1}.
    On this representation, the opposite faces of the octahedron need to be identified.
\end{Ex}

    \begin{figure}
        \centering
        \includegraphics{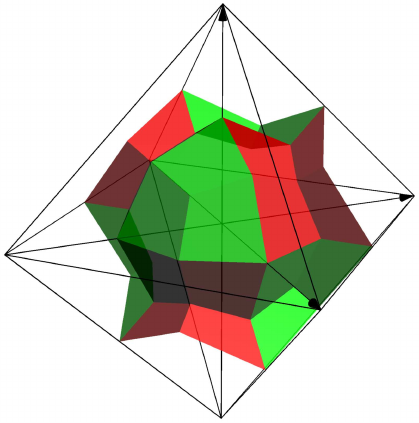}
        \caption{The T-surface $X_{\mathcal E_2}$ of Example \ref{ex2}}
        \label{n3k1}
    \end{figure}

\begin{Ex}\label{ex3}
    For  $\Delta_3 = [a,b,c,d]$ with vertices $a=(0,0,0)$, $b=(1,0,0)$, $c=(0,1,0)$ and $d=(0,0,1)$,
    a real phase structure $\mathcal E_3$ on the 2-skeleton of $\Delta_3$ is given by
    \begin{itemize}
        \item $\mathcal E_3([a,b,c]) = \{(+,+,+), (+,+,-)\}$
        \item $\mathcal E_3([a,b,d]) = \{(+,+,+), (+,-,+) \}$
        \item $\mathcal E_3([a,c,d]) = \{(+,+,-), (-,+,-) \}$
        \item $\mathcal E_3([b,c,d]) = \{(+,-,+), (-,+,-) \}$
    \end{itemize}
    The T-curve $X_{\mathcal E_3}$ is represented in Figure \ref{n3k2}.
    On this representation, the opposite faces of the octahedron need to be identified.
           
\end{Ex}

    \begin{figure}[t]
        \centering
        \includegraphics{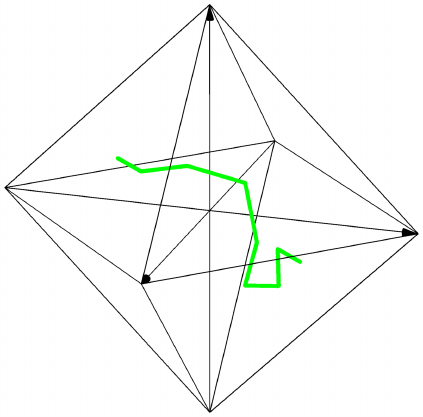}
        \caption{The T-curve $X_{\mathcal E_3}$ of Example \ref{ex3}}
        \label{n3k2}
    \end{figure}

\begin{Ex}\label{signe}

In codimension 1, a real phase structure is equivalently defined by a sign distribution $\mu \colon F_0(\Gamma) \to \{-, +\}$ on the vertices of $\Gamma$.
Such a sign distribution $\mu$ extends to vertices of $\widetilde{\Gamma}$ by the following rule: 
\[ \mu(s(v)) = (-1)^{s \cdot v}\mu(v) \]
where $s \cdot v = \sum_{i=1}^n s_i v_i$.
We then define $\mathcal E_\mu$ by saying that $s \in \mathcal E_\mu([a,b])$ if and only if $\mu(s(a)) \neq \mu(s(b))$. 

For any sign distribution $\mu$, we obtain a 1-real phase structure $\mathcal E_\mu$. Conversely, a real phase structure of codimension 1 determines, up to inversion of all the signs, a sign distribution $\mu \colon F_0(\Gamma) \to \{-, +\}$.
See \cite[Example 4.3]{brugalle2024combinatorial} for the details of this bijection.

For example, the real phase structure $\mathcal E_1$, of example \ref{ex1}, arises from the sign distribution $\mu$ (or $-\mu$), where $\mu$ is defined by $\mu(a)=+$, $\mu(b)=-$ and $\mu(c) =-$.

\end{Ex}

\subsection{Real toric variety}

In this subsection, we relate the construction of T-manifolds of codimension 0 and 1 to real algebraic geometry.

For a lattice polytope $\Delta$ of $\mathbb R^n$ with lattice points  $\Delta \cap \mathbb Z^n = \{m_0, \dots, m_N\}$,
 we can define an application
\begin{equation*}
     \phi^n_\Delta \colon \begin{aligned}
         (\mathbb C^*)^n &\longrightarrow  \mathbb C P^N\\
      (t_1, \dots, t_n) &\longmapsto  [t^{m_0} \colon \dots \colon t^{m_N}]
     \end{aligned}
 \end{equation*} where $[u_0 : u_1 : \dots : u_N]$ are the homogeneous coordinates on $\mathbb CP^N$ and $t^{m_i} := \allowbreak t_1^{m_{i,1}} \dots t^{m_{i,n}}$.

Denote by $\mathbb C \mathring X_\Delta$ the image of $\phi_\Delta$.
The Zariski closure  of the image $\mathbb C \mathring X_\Delta$ in $\mathbb C P^N$ is the complex toric variety $\mathbb C X_ \Delta$.

\begin{Ex}
    For the standard $n$-simplex 
    \[ \Delta_n = \{x \in  \mathbb R^n | x_i \geq 0 \text{ and } x_1+\dots +x_n \leq 1\}, \]
    the image $\mathbb C \mathring X_{\Delta_n}$ is the complement ${\{[u_0: \dots : u_n] | u_i \in \mathbb C^* \}}$  of the coordinate hyperplanes in $\mathbb C P^n$ and
    its Zariski closure in $\mathbb C P^n$ is \[ \mathbb CX_{\Delta_n} = \mathbb C P^n .\]
\end{Ex}

\begin{Rem}
    For any positive integer $m$, we denote by $m\Delta$ the $m$-dilation $\{mx| x \in \Delta\}$ of $\Delta$.
    A polytope and its dilations have isomorphic (as algebraic varieties) toric varieties: $\mathbb CX_{m\Delta} = \mathbb CX_{\Delta}$.
\end{Rem}

For a face $F$ of $\Delta$ with  $M+1$ lattice points,  we can see $\mathbb CP^M$ as the subset  $\{[u_0 : \dots : u_N] | u_i =0 \text{ if  } m_i \notin F \cap \mathbb Z^n \}$ of $\mathbb CP^N $.
This defines an embedding $\mathbb C X_F \hookrightarrow \mathbb CX_\Delta$.
More precisely, we have the stratification $\mathbb CX_\Delta = \bigcup_{F \text{ face of } \Delta} \mathbb C \mathring X_F$.

\begin{lemma}[Theorem 2.4.3,\cite{GBBIB3053}]
    The toric variety $\mathbb CX_\Delta$ is smooth if and only if the polytope $\Delta$ is smooth.
\end{lemma}

The real part of the toric variety $\mathbb CX_\Delta$  is denoted $\mathbb RX_\Delta$ and
is defined as the intersection of $\mathbb CX_\Delta$ with $\mathbb RP^N \subset \mathbb CP^N$.

\begin{theo}[Theorem 5.3 and 5.4, Chapter 11, \cite{GKZ}]\label{realtoric}
    There is a homeomorphism from $\widetilde \Delta$ to the real part $\mathbb R X_\Delta$.
    This homeomorphism is stratified in the following sense: for any face $F$ of $\Delta$, the gluing $\widetilde F \subset \widetilde \Delta$ is sent to $\mathbb R X_F \subset \mathbb R X_\Delta$.
    In addition, for $s \in \mathcal Q^n$, the copy $s(\Delta)$ is sent to the orthant 
    \[ \mathbb RX_\Delta \cap \{ [u_0: \dots:u_N] \mid \forall i, \; s_i u_i  \ge 0 \} \] associated to $s$.
\end{theo}

It follows from Example \ref{ex0} that T-manifolds of codimension 0 are toric varieties:

\begin{coro}
    Let $\Delta$ be a lattice polytope, let $\Gamma$ be a triangulation of $\Delta$, and let $\mathcal E$ a real phase structure on the 0-skeleton of $\Gamma$,
    then $X_\mathcal E$ is homeomorphic to $\mathbb RX_\Delta$.
\end{coro}

For a map $\nu \colon \Delta \cap \mathbb Z^n \to \mathbb R$, define the polyhedron 
\[ \Delta_\nu = \Conv\{(x, h) \in (\Delta \cap \mathbb Z^n) \times \mathbb R| h \geq \nu(x)\} \subset \mathbb R^{n+1}.\]
The orthogonal projection of the bounded faces of $\Delta_\nu$ on $\Delta$ defines a subdivision $\Gamma_\nu$ of $\Delta$.
Subdivisions obtained in this way are called convex (or regular, or coherent).

The relationship between algebraic geometry and patchworking comes from Viro's theorem. Here we reformulate it in terms of real phase structures.

\begin{theo}[Viro's theorem \cite{Viro84}]
    If $\Gamma_\nu$ is a convex triangulation on $\Delta$ and $\mathcal E$ a 1-real phase structure on $\Gamma_\nu$,
    there exists an algebraic hypersurface $H$ of $\mathbb R X_\Delta$ such that the pair $(\mathbb RX_\Delta, H)$ is homeomorphic to $(\widetilde \Delta, X_{\mathcal E})$.
    This homeomorphism is stratified in the sense of Theorem \ref{realtoric}.
\end{theo}

    If $\mu$ is a sign distribution corresponding to $\mathcal E$ (see Example \ref{signe}), the hypersurface $H$ is given by the polynomial equation $\{P_t =0\} \subset \mathbb RX_\Delta$, for sufficiently small $t>0$, where
    \[ P_t(x) = \sum_{a \in F_0(\Gamma_\nu)} \mu(a) t^{\nu(a)} x^a .\]
    Such a polynomial is called a \emph{Viro polynomial}.

\section{Real phase structures on the standard simplex}

The goal of this section is to prove the two following theorems.
\begin{theo}\label{nbfacet}
    Let $\Delta$ be a lattice polytope of dimension $n$, let $\Gamma$ be an unimodular triangulation of $\Delta$ and let $\mathcal E$ be a real phase structure on the $k$-skeleton of $\Gamma$.
    The number of maximal cells in the canonical cell decomposition of $X_\mathcal E$ is
    \[ \sum_{i=k}^n \binom{n}{k} \Vol(\Delta).\]
\end{theo}

\begin{theo}\label{nbsimplex}
    Let $\Delta$ be a lattice polytope of dimension $n$, let $\Gamma$ be an unimodular triangulation of $\Delta$ and let $\mathcal E$ be a real phase structure on the $k$-skeleton of $\Gamma$.
    The number of maximal and simplicial cells in the canonical cell decomposition of $X_\mathcal E$ is at most
    \[ \frac{2}{k+1} \binom{n+1}{k} \Vol(\Delta).\]
\end{theo}

Here an $m$-cell is said simplicial if it has exactly $m+1$ facets.

Recall that a full-dimensional simplex $\sigma \in \Gamma$ contributes to a full-dimensional cell $C_{\mathcal E}(s(\sigma))$ of $X_\mathcal E$ for each orthant $s \in \mathcal E(\sigma)$.
Thus, we will enumerate the orthants of $\mathcal E(\sigma)$ for $\sigma$ a unimodular simplex.

In this section, the lattice polytope $\Delta$ considered is a unimodular lattice simplex of dimension $n$ in $\mathbb R^n$,
 and $\mathcal E$ denotes a real phase structure on the $k$-skeleton of $\Delta$.
Up to a unimodular transformation, one can always choose the standard $n$-simplex $\Delta_n = \Conv(0, e_1, \dots, e_n) \subset \mathbb R^n$ where $(e_1, \dots, e_n)$ is the canonical basis of $\mathbb R^n$.
The polytope $\Delta_n$ has a unique triangulation, whose simplices are precisely the faces of $\Delta_n$.

In the dual setting of fans, real phase structures have been studied in \cite{rau2022real}.
In particular, it was shown there that real phase structures on a fan are in bijection with the orientations of the matroid associated to the fan.
By the topological representation theorem of Folkman and Lawrence \cite{folkman1978oriented}, every oriented matroid is realizable by arrangement of pseudohyperplanes in a real projective space.

The arrangement corresponding to a real phase structure $\mathcal E$ can be recovered from the associated T-variety $X_\mathcal E \subset \widetilde{\Delta_n}$:
For each facet $F$ of $\Delta$, $\widetilde{F}$ intersects $X_\mathcal E$ as a hypersurface $H_F$ in $X_\mathcal E$.
It turns out that the collection of the $H_F$ in $X_\mathcal E$ defines an arrangement of pseudohyperplanes in a projective space.
In particular, the T-manifold $X_\mathcal E$ is homeomorphic to $\mathbb RP^{n-k}$.

This provides the following geometrical intuition of real phase structures on unimodular simplex: 
the structure $\mathcal E$ can be thought as an arrangement $\mathcal A$ of $n+1$ (pseudo)hyperplanes in $\mathbb RP^{n-k}$. The chambers (i.e. the full-dimensional region) $c$ of the arrangement $\mathcal A$ correspond to orthants $s \in \mathcal E(\Delta_n)$,
 and the vertices of the chambers $c$ correspond to $k$-faces $\sigma$ of $\Delta_n$ such that $s \in \mathcal E(\sigma)$.

The Figure \ref{linearrangement} represents the arrangement of four pseudolines in $\mathbb RP^2$ obtained from the real phase structure $\mathcal E_2$ of Example \ref{ex2}.
Each of the four black lines corresponds to a facet of $\Delta_3$, and they delimit seven chambers corresponding to the non-empty orthants of $X_{\mathcal E_2}$.

\begin{figure}
    \includegraphics{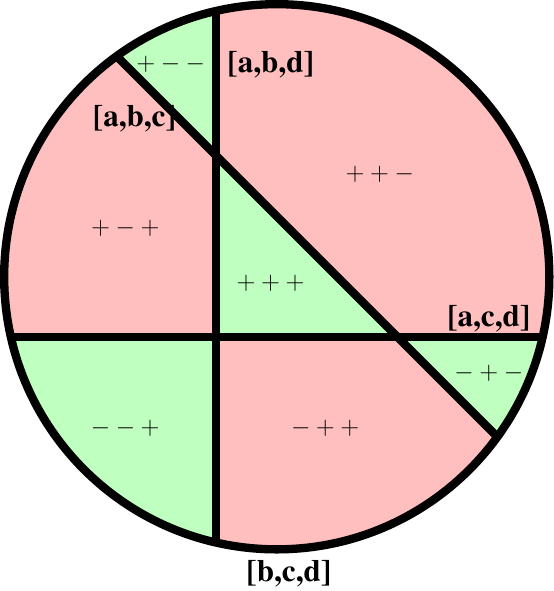}
    \caption{Arrangement of pseudolines in $\mathbb RP^2$ associated to $\mathcal E_2$}
    \label{linearrangement}
\end{figure}

With this correspondence, estimating the number of orthants in $\mathcal E(\Delta)$ is equivalent to counting the chambers in a pseudohyperplanes arrangement.
The answer of this problem is well known in the theory of oriented matroids.
However, we are here only interested to the simple case of uniform matroids.
Thus, to keep the presentation self-contained and to avoid relying on the full machinery of oriented matroids and topological representation theory, we prefer to use direct combinatorial arguments based on the elementary properties of real phase structures.

\subsection{Restrictions and projections of real phase structure}

For a facet $F$ of $\Delta_n$, the orthogonal space $T_2^\perp(F)$ is a one-dimensional $\mathbb F_2$-vector space. We denote its unique non-zero element by $e_F$.
Since $\Delta_n$ is a unimodular simplex, it follows that for any strict subset $I \subsetneq F_{n-1}(\Delta_n)$, the vectors $(e_{F})_{F \in I}$ are linearly independent.
The only non-trivial relation between the $(e_F)_{F \in F_{n-1}(\Delta_n)}$ is:
\[ \sum_{F \in F_{n-1}(\Delta_n)} e_F=0.\]
 
Given a lattice simplex, the quotient $\mathcal Q^n / T_2^\perp(\sigma)$ is an affine space of direction the dual space $T_2^{\vee}(\sigma) = \Hom(T_2(\sigma), \mathbb F_2)$, since this vector space is canonically isomorphic to the quotient $(\mathbb F_2^n)^{\vee}/T_2^\perp(\sigma)$.
Let $\pi_{\sigma} \colon \mathcal Q^n \to \mathcal Q^n / T_2^\perp(\sigma)$ be the projection map.

\begin{defi}
    Let $\mathcal E$ be a real phase structure on the $k$-skeleton of $\Delta_n$ and let  $F$ be a facet of $\Delta_n$.

    If $k<n$, the restriction $\mathcal E_{|F}$ is defined by $\mathcal E_{|F}(\sigma) = \pi_F( \mathcal E(\sigma) )$ for each $k$-simplex $\sigma$ of $F$.

    If $k>0$, the projection $\mathcal{E}^F$ is defined by $\mathcal E^F(\sigma \cap F) = \pi_F( \mathcal E(\sigma) )$ for each $k$-simplex $\sigma$ of $\Delta_n$ not contained in $F$.
\end{defi}

The restriction and projection operations on real phase structure are the equivalent of deletion and contraction for (oriented) matroids.
In the case of real phase structure on fan, they are defined and studied in \cite{rau2023realphase}.

\begin{prop}
    Let $\mathcal E$ be a real phase structure on the $k$-skeleton of $\Delta_n$, and let $F$ be a facet of $\Delta_n$.
    \begin{itemize}
        \item If $k<n$, the restriction $\mathcal{E}_{|F}$ is a real phase structure on $F$ of codimension $k$.
        \item If $k>0$, the projection $\mathcal{E}^F$ is a real phase structure on $F$ of codimension  $k-1$.
    \end{itemize}
   \end{prop}

\begin{proof}

    \begin{itemize}
    \item Let $\sigma$ be a $k$-simplex of $F$. The image $\overrightarrow{\pi_F}(T_2^\perp(\sigma))$ is the dual of $T_2(\sigma)$ viewed as a subspace of $T_2(F)$.
    Thus, the affine space $\mathcal E_{|F}(\sigma) = \pi_F( \mathcal E(\sigma) )$ is directed by the dual of $T_2(\sigma) \subset T_2(F)$.
    
    It remains to show the parity condition.

    Let $s \in \mathcal Q^n$. The fiber $\pi_F^{-1}(\pi_F(s))$ consists of two elements: \[ \{s, s+e_F\} = s + T_2^\perp(F).\]
    For a $k$-simplex $\sigma$ in  $F$, since $T_2^\perp(F) \subset T_2^\perp(\sigma)$, we have $s \in \mathcal E(\sigma)$ if and only if $s + e_F \in \mathcal E(\sigma)$.
    Therefore, the condition $\pi_F(s) \in \mathcal E_{|F} (\sigma)$ is equivalent to $s \in \mathcal E(\sigma)$.
    Thus, the parity condition for $\mathcal E_{|F}$ follows directly from the parity condition of $\mathcal E$. \newline 
    \item Let $\sigma$ be a $k$-simplex of $\Delta_n$ not contained in $F$.
    Since $\Delta_n$ is unimodular, we have $T(\sigma)+T(F) = \mathbb Z^n$, and consequently,
    \[T_2^\perp(\sigma \cap F) = T_2^\perp(\sigma) + T_2^\perp(F).\]
    Thus, we have the equality $\overrightarrow{\pi_F}(T_2^\perp(\sigma \cap F)) = \overrightarrow{\pi_F}(T_2^\perp(\sigma))$.
    And $\mathcal E^F(\sigma \cap F) = \pi_F( \mathcal E(\sigma) )$ is an affine space directed by the dual of $T_2(\sigma \cap F)$ inside $T_2(F)$.
    
    It remains to show the parity condition. Let $s \in \mathcal Q^n$ such that $\pi_F(s) \in \mathcal E^F(\sigma \cap F) =  \pi_F( \mathcal E(\sigma) )$.
    Either $s \in \mathcal E(\sigma)$ or $s+e_F \in \mathcal E(\sigma)$.
    But the two possibilities can not both occur simultaneously, as this would imply that $e_F$, and thus $T_2^\perp(F)$, is contained in the direction $T_2^\perp(\sigma)$ of $\mathcal E(\sigma)$ contradicting that $\sigma$ is not contained in $F$.
    
    The number of $k$-simplices $\sigma$ such that $\pi_F(s) \in \mathcal{E}^F(\sigma \cap F)$ is thus the sum of the number of $k$-simplices $\sigma$ such that $s \in \mathcal{E}(\sigma)$ and the number of those such that $s+e_F \in \mathcal{E}(\sigma)$.
    The parity condition of $\mathcal E^F$ for the orthant $\pi_F(s)$ is then a consequence of the parity condition of $\mathcal E$ for the orthant $s$ and the orthant $s+e_F$, since
    the sum of two even numbers is even.
    \end{itemize}
\end{proof}

The T-manifolds $X_{\mathcal E_{|F}}$ and $X_{\mathcal E^F}$ are obtained from $X_\mathcal E$ in the following way.

The T-manifold $X_{\mathcal E_{|F}}$ is a subcomplex of $X_{\mathcal E}$, defined as the intersection $X_{\mathcal E} \cap \widetilde F$.
Indeed, we have observed that for a $k$-simplex $\sigma$ of $F$, we have $\pi_F(s) \in \mathcal E_{|F}(\sigma)$  if and only if $s \in \mathcal E(\sigma)$.
Thus, by identifying $C_{\mathcal E}(\pi_F(s)(\sigma)) \subset \widetilde F$ with $C_{\mathcal E}(s(\sigma)) \subset \widetilde \Delta_n$,
we conclude that $X_{\mathcal E_{|F}}$ is the subcomplex of $X_{\mathcal E}$ composed of all the cells lying in $\widetilde F$.

We now describe the T-manifold $X_{\mathcal E^F}$. Denote by $v_F$ the vertex of $\Delta_n$ opposite to $F$. For each point $p$ in $\Delta$
distinct from $v_F$, there is a unique point $\varphi(p)$ of $F$ lying on the line define by $v_F$ and $p$.
The application $\varphi$ extends to an application $\widetilde \varphi \colon \widetilde \Delta_n \setminus \{v_F\} \to \widetilde F$,
where $v_F$ denotes abusively the common image of all the symmetric copies $s(v_F)$ of $v_F$ in the gluing $\widetilde \Delta$.
The image of the barycentric subdivision of $\widetilde \Delta_n$ under  $\widetilde \varphi$ is the barycentric subdivision of $\widetilde F$.
More precisely, the image of the barycenter of a face $s(\tau)$ under $\widetilde \varphi$ is the barycenter of $s(\tau) \cap \widetilde F$. 
With this description, we see that the T-manifold associated to the projection $X_{\mathcal E^F}$ in $\widetilde F$ is the image $\varphi(X_\mathcal E)$.

\begin{Ex}
    Consider the real phase structure $\mathcal E_1$, $\mathcal E_2$ and $\mathcal E_3$ from Examples \ref{ex1}, \ref{ex2} and \ref{ex3}.
    Identifying $\Delta_2$ as the facet $[a,b,c]$ of $\Delta_3$,
    we have $\mathcal E_1 = \mathcal E_{2|\Delta_2}$ and $\mathcal E_1 = \mathcal E_3^{\Delta_2}$.
\end{Ex}

\subsection{Number of orthants of $\mathcal E(\Delta_n)$}

In this subsection, we develop the material necessary to prove of Theorem \ref{nbfacet} on the number of maximal cells.

\begin{lemma}\label{necklace}
    Let $\mathcal E$ be a $k$-real phase structure on $\Delta_n$ (with $k<n$), and let $\tau$ be a $(k+1)$-simplex of $\Delta_n$.
    There exists a cyclic order $\sigma_0, \dots, \sigma_{k+1}$ on the $k$-simplices of $\tau$
    and affine subspaces $\mathcal B_i \subset \mathcal Q^n$ directed by $T_2^\perp(\tau)$ such that
    $\mathcal E(\sigma_i) = \mathcal B_i \cup \mathcal B_{i+1}$.
    Moreover, if $F_i$ is a facet of $\Delta_n$ such that $\sigma_i = \tau \cap F_i$, then $\mathcal B_{i+1} = \mathcal B_i +e_{F_i}$.
\end{lemma}
 
Here the indices $i$ are treated modulo $k+2$.

\begin{proof}
    We prove the first part by induction on $n-k$.

    If $n-k=1$, the simplex $\tau$ is $\Delta_n$ and the $k$-simplices are the facets $F$ of $\Delta_n$.
    Recall that spaces $\mathcal E(F)$ are lines in $\mathcal Q^n$ directed by $e_F$.
    Consider all these lines corresponding to facet $F$ of $\Delta_n$.
    By the parity condition each of the two points of any such line is contained in at least one other line.
    Suppose the lines corresponding to $F_1, \dots, F_m$ form a cycle, that is, there are points $s_i \in \mathcal Q^n$ such that $\mathcal E(F_i) = \{s_i, s_{i+1}\}$.
    Then $\sum_{i=1}^m e_{F_i} = \sum_{i=1}^m \overrightarrow{s_is_{i+1}}=0$.
    However, the only non-trivial relation between the vectors $e_{F}$ is $\sum_{F \in F_{n-1}(\Delta_n)} e_F =0$.
    Thus, the lines form a unique cycle.
    In other words, there exists a cyclic order $F_0, \dots F_n$ and points $s_i \in \mathcal Q^n$ such that $\mathcal E(F_i) = \{s_i, s_{i+1}\}$,
    which proves the claim in this case.

    If $n-k>1$, consider a facet $F$ of $\Delta_n$ containing $\tau$.
    We apply the heredity property to $\mathcal E_{|F}$ in $F$.
    We then obtain a cyclic order $\sigma_0, \dots, \sigma_{k+1}$
    and affine subspaces $\mathcal C_i$ of $\mathcal Q^n /T_2^{\perp}(F)$ directed by $T_2^{\perp}(\tau)/T_2^{\perp}(F)$
    such that $\mathcal E_{|F}(\sigma_i) = \mathcal C_i \cup \mathcal C_{i+1}$.
    Setting $\mathcal B_i = \pi_F^{-1}(\mathcal C_i)$,
    since $e_F \in  T_2^\perp(\sigma_i)$, we obtain $\mathcal E(\sigma_i) = \pi_F^{-1}(\mathcal E_{|F}(\sigma_i))= \mathcal B_i \cup \mathcal B_{i+1}$.

    For the second part of the lemma, note that $e_{F_i}$ is in the direction of $\mathcal E(\sigma_i)$ but not in the direction of $\mathcal B_i$ and $\mathcal B_{i+1}$.
    Hence, there exist two orthants $s \in \mathcal B_i$ and $s' \in \mathcal B_{i+1}$ such that $e_{F_i} = s' - s$.
    Since $\mathcal B_i$ and $\mathcal B_{i+1}$ are parallel, it follows that $\mathcal B_{i+1} = \mathcal B_i + e_{F_i}$.
\end{proof}

\begin{coro}\label{adjacency}
    Let $\mathcal E$ be a $k$-real phase structure on $\Delta_n$ and let $F$ be a facet of $\Delta_n$. Let $\sigma$ and $\sigma'$ be $k$-faces of a $(k+1)$-simplex $\tau$ of $\Delta_n$.
    
    For $s \in \mathcal Q^n$, if $s \in \mathcal E(\sigma)$ and $s +e_F \in \mathcal E(\sigma')$, then either $s+e_F \in \mathcal E(\sigma)$ or $s \in \mathcal E(\sigma')$. In particular, $s \in \mathcal E(F)$.
\end{coro}
\begin{proof}
    By Lemma \ref{necklace}, we have a cyclic ordering of the $k$-faces of $\tau$: $ \sigma_0, \sigma_1, \dots, \sigma_{k+1}$ and affine subspaces $\mathcal B_0, \mathcal B_1, \dots, \mathcal B_{k+1}$ such that
    $\mathcal E(\sigma_j) = \mathcal B_j \cup \mathcal  B_{j+1}$.
        Without loss of generality, assume that $s \in \mathcal B_{0} \subset \mathcal E(\sigma_0)$ and $s+e_F \in \mathcal B_{i+1} \subset \mathcal E(\sigma_i)$.
        
        We partition facets of $\Delta_n$ in two categories: 
    \begin{itemize}
        \item $I_1$: the facets $G$ that contain $\tau$, for which $T_2^\perp(\tau) = \spn_{G \in I_1}(e_G)$
        \item $I_2$: the facets $F_j$, for $0 \leq j \leq k+1$, that intersect $\tau$ in the $k$-simplex $\sigma_j$
        for which $T_2^\perp(\sigma_{j})=T_2^\perp(\tau) \oplus T_2^\perp(F_{j})=T_2^\perp(\tau) \oplus \{0, e_{F_{j}} \}$.
    \end{itemize}
    
    Since $\mathcal B_{i+1} = \mathcal B_i + e_{F_i} = \dots = \mathcal B_0 + e_ {F_0} + \dots + e_{F_i}$,
        we have $s' = s + e_ {F_0} + \dots + e_{F_i} \in \mathcal B_{i+1}$.
        Hence, the difference $e_ {F_0} + \dots + e_{F_i} - e_F$ between $s'$ and $s+e_F$ is in the direction $T_2^\perp(\tau)= \spn_{G \in I_1}(e_G)$ of $\mathcal B_{i+1}$.
        We deduce that 
        \[ e_F = e_{F_0} + e_{F_{1}} + \dots + e_{F_{i}} + u \]
    where $u \in T_2^\perp(\tau) = \spn_{G \in I_1}( e_G )$. This is a linear relation among the $e_F$'s, for $F$ facet of $\Delta$.
    However, the only nontrivial relation between the $e_F$'s is:
    \[ \sum_{F \in F_{n-1}(\Delta_n)} e_F=0.\]
    Thus, there are three cases:
    \begin{itemize}
        \item $F \in I_1$. In this case, we have both $s + e_F\in \mathcal E(\sigma_0)$ and $s \in \mathcal E(\sigma_i)$.
        \item $i=0$  and $u=0$. In this case, we have $\sigma_{0} \subset F_{0} =F$ and $s +e_F \in \mathcal E(\sigma_{0})$.
        \item $i=k$ and $u = \sum_{G \in I_1}{e_G}$. In this case, we have $F = F_{k+1}$  and \\ $s = (s+e_F)-e_F \in \mathcal E(\sigma_{k+1})= \mathcal E(\sigma_{i})$.
    \end{itemize}
\end{proof}

\begin{defi}
    A path between two $k$-simplices $\sigma$ and $\sigma'$ in $\Delta_n$ is
    a sequence of $k$-simplices $\sigma^{(0)}= \sigma, \sigma^{(1)}, \dots, \sigma^{(l)}=\sigma'$ such that every consecutive couple $(\sigma^{(i)}, \sigma^{(i+1)})$ share a common $(k-1)$-simplex (or equivalently, lie on a common $(k+1)$-simplex).
\end{defi}

\begin{lemma}\label{pathsimplex}
    Let  $\mathcal E$ be a $k$-real phase structure on $\Delta_n$. For any $\sigma, \sigma' \in F_k(\Delta_n)$ and  $s \in \mathcal Q^n$,
    if $s \in \mathcal{E}(\sigma) \cap \mathcal{E}(\sigma')$, then there exist $\sigma^{(0)}, \sigma^{(1)}, \dots \sigma^{(l)}$ a path between $\sigma$ and $\sigma'$ such that $s \in \mathcal E(\sigma^{(i)})$ for all $i$ between $0$ and $l$.
\end{lemma}
 \begin{proof}
    For $k=n$, necessarily $\sigma = \Delta_n = \sigma'$ and the result is trivial.
    
    For $k=0$, a vertex $v$ satisfies $T_2^\perp(v)=(\mathbb F_2^n)^{\vee}$, so $\mathcal E(v) = \mathcal Q^n$ and any path works.
     
     Otherwise, we proceed  by induction on $n$:
      We may suppose $\sigma$ and  $\sigma'$ share at least one vertex.
         
         Indeed, if they do not, select a vertex $v$ of $\sigma'$ and let $\tau$ be the $(k+1)$-simplex containing $\sigma$ and having $v$ as one of its vertices.
         
         The parity condition of $\mathcal E$ on $\tau$ and $s$ ensures that there exists another $k$-simplex $\sigma'' \neq \sigma$ inside $\tau$ such $s \in \mathcal E(\sigma'')$.
         
         Now $\sigma''$ and $\sigma'$ share the vertices $v$. So up to replacing $\sigma'$ by $\sigma''$, we suppose $\sigma$ and  $\sigma'$ share a vertex.

      Let $F$ be the facet dual to a shared vertex of $\sigma$ and $\sigma'$. Consider the projection $\mathcal E^F$. We have 
         \[ \pi_F(s) \in \mathcal{E}^F(\sigma \cap F) \cap \mathcal{E}^F( \sigma'\cap F).\]
         
         Thus, by induction hypothesis, there exists a path $\sigma^{(0)} = \sigma, \dots, \sigma^{(l)}= \sigma'$ such that $\pi_F(s) \in \mathcal E^F(\sigma^{(i)} \cap F)$. 
         
      By definition of $\mathcal E^F$, it implies that, for each $i$, either $s$ or $s+e_F$ is in $\mathcal E(\sigma ^{(i)})$.
         
         Let $m$ (resp. $M$) the minimal (resp. maximal) index such $s \notin \mathcal E(\sigma^{(i)})$.

          Since $s \in \mathcal E(\sigma^{(m-1)})$ and $s + e_F \in \mathcal E(\sigma^{(m)})$ and since $\sigma^{(m-1)}$ and $\sigma^{(m)}$ lie in a common $(k+1)$-face, it follows that $s+e_F \in \mathcal E(\sigma^{(m-1)})$ by Corollary \ref{adjacency}.
          
          Consequently, the simplex $\sigma^{(m-1)}$ must be contained in $F$. Similarly, the simplex $\sigma^{(M+1)}$ must also be contained in $F$.
          
          Consider the restriction $\mathcal E_{|F}$.
          
          We have $\pi_F(s) \in \mathcal E_{|F}(\sigma^{(m-1)}) \cap \mathcal E_{|F}(\sigma^{(M+1)})$. 
          
          By induction hypothesis, there exists a path $\widetilde{\sigma}^{(1)}, \dots, \widetilde \sigma^{(L)}$ between $\sigma^{(m-1)}$ and $\sigma^{(M+1)}$in $F$ such that $\pi_F(s) = \mathcal E_{|F}(\widetilde \sigma^{(i)})$.

            Combining the paths, we obtain the desired path:
            \[\sigma^{(0)}, \sigma^{(1)}, \dots, \sigma^{(m-1)} =\widetilde \sigma^{(1)}, \dots, \widetilde \sigma^{(L)} =\sigma^{(M+1)}, \dots, \sigma^{(l)}\] 
            This completes the proof.
  \end{proof}

\begin{prop}\label{nborthant}
     Let $\Delta_n$ be a $n$-dimensional lattice simplex and $\mathcal E$ a $k$-real phase structure on $\Delta_n$. Then
     \[ |\mathcal E(\Delta_n)| = \sum_{i=k}^{n} \binom{n}{i}.\]
\end{prop}

\begin{proof}
    For $k=0$, we have $\mathcal E(v) = \mathcal Q^n$ for any vertex, so 
    \[ |\mathcal E( \Delta_n)| = |\mathcal Q^n|=2^n=\sum_{i=0}^{n} \binom{n}{i}.\]
    For $k=n$, the vector space $T_2^\perp(\Delta_n)$ is trivial, hence $\mathcal E(\Delta_n)$ consists of a single point.
    Therefore,
    \[ |\mathcal E(\Delta_n)| = 1 = \binom{n}{n}. \]
    
    Otherwise, we proceed by induction on the dimension $n$ of $\Delta_n$.
    The base case $n=0$ is already included in the case $k=0$ above, so we may assume $n>k > 0$

    Fix a facet $F$ of $\Delta_n$.
    We divide each $s \in \mathcal E(\Delta_n)$ in two categories depending on their relationship with $F$.
    An orthant $s$ belongs to the first category if $s+e_F \in \mathcal E(\Delta_n)$; otherwise, $s$ belongs to the second category.

    In the first category, we have $s \in \mathcal E(F)$.
    Indeed, suppose $s \in \mathcal E(\sigma)$ and $s+e_F \in \mathcal E(\sigma')$ for some $k$-simplices $\sigma, \sigma'$ of $\Delta_n$.
    If either $\sigma$ or $\sigma'$ is in $F$, we are done.
    Otherwise, consider $\pi_F(s) \in \mathcal E^F(\sigma \cap F) \cap \mathcal E^F(\sigma' \cap F)$.
    By Lemma \ref{pathsimplex}, we have a path $\sigma^{(1)}, \dots, \sigma^{(l)}$ such that $\pi_F(s) \in \mathcal E^F(\sigma^{(i)} \cap F)$ for each $i$.
    Since $s \in \mathcal E(\sigma)$ and  $s+e_F \in \mathcal E(\sigma')$, we have an index $i$ such that $s \in \mathcal E(\sigma^{(i)})$ and $s +e_F \in \mathcal E(\sigma^{(i+1)})$.
    And the Corollary \ref{adjacency} implies that $s \in \mathcal E(F)$.

    In the second category, we have $s \notin \mathcal E(F)$. So $\pi_F(s) \notin \mathcal{E}_{|F}(F)$ but $\pi_F(s) \in \mathcal E^F(F)$.

    Denote $a$ the number of pair $\{s, s+e_F\}$ of the first category and $b$ the number of element $s$ of the second category.
    The previous discussion shows that $|\mathcal E(\Delta_n)| = 2a + b$, $|\mathcal E^F(F)| = a+b$ and $|\mathcal E_ {|F}(F)| = a$.
    Hence, $|\mathcal E(\Delta_n)| = |\mathcal E_{|F}(F)| + |\mathcal E^F(F)|$.
    Now we can apply the recursion hypothesis on $\mathcal E^F$ and $\mathcal E_{|F}$: 
    \[ |\mathcal E_{|F}(F)| = \sum_{i=k}^{n-1} \binom{n-1}{i} \text{ and } |\mathcal E(F)| = \sum_{i=k-1}^{n-1} \binom{n-1}{i}.\] 
    Now using Pascal's formula:
    \begin{equation*}
    \begin{split}
        |\mathcal E(\Delta_n)| & = |\mathcal E_{|F}(F)| + |\mathcal E^F(F)| \\
                        & =  \sum_{i=k}^{n-1} \binom{n-1}{i} + \sum_{i=k-1}^{n-1} \binom{n-1}{i} \\
                        & = \sum_{i=k}^{n-1} \left[ \binom{n-1}{i}+\binom{n-1}{i-1} \right] + \binom{n-1}{n-1} \\
                        & = \sum_{i=k}^{n-1} \binom{n}{i}  + \binom{n}{n} \\
                        & = \sum_{i=k}^n \binom{n}{i}
    \end{split}
    \end{equation*}
\end{proof}

\begin{proof}[Proof of Theorem \ref{nbfacet}]
    For each $n$-simplex $\sigma$ of $\Gamma$,
    the simplex $s(\sigma)$ contains a maximal cell of $X_{\mathcal{E}}$ if and only if $s \in \mathcal{E}(\sigma)$.
    Thus, by Theorem \ref{nborthant}, each $n$-simplex of $\gamma$ contributes to $\sum_{i=k}^{n} \binom{n}{i}$ maximal cells of $X_\mathcal{E}$.
    Since $\Gamma$ is unimodular, it contains a number of $n$-simplices equal to $\Vol(\Delta)$.

    Hence, the total number of maximal cells in the canonical decomposition of $X_\mathcal E$ is $\sum_{i=k}^{n} \binom{n}{i} \Vol(\Delta)$.
\end{proof}

If $s \in \mathcal E(\Delta_n)$, $s$ is called a non-empty orthant.

\subsection{Number of simplicial orthants}

In this subsection, we detail the necessary material to obtain the proof of the Theorem \ref{nbsimplex} on the number of simplicial cells.

\begin{lemma}\label{minorthant}
    Let $s \in \mathcal E(\Delta_n)$. Then one has
    \[|\{ \sigma \in F_k(\Delta_n)| s \in \mathcal E(\sigma) \}| \geq n-k+1 .\]
    In case of equality, the simplices of $\{ \sigma \in F_k(\Delta_n)| s \in \mathcal E(\sigma) \}$ are all adjacent but no three of them can lie in a common $(k+1)$-simplex.
\end{lemma}
\begin{proof}
    Since $s \in \mathcal E(\Delta_n)$, there exists a $k$-simplex $\sigma$ such that $s \in \mathcal E(\sigma)$.
    Now consider any $(k+1)$-simplex $\tau$ that contains $\sigma$. The parity condition ensures that there exists another $k$-simplex $\sigma'$ such that $s \in \mathcal E(\sigma')$.
    Each choice of $\tau$ corresponds to a choice of a vertex from $\Delta_n$ that is not already in $\sigma$. Since $\Delta_n$ has $n+1$ vertices, and $\sigma$ has  $k+1$ vertices, there are $n-k$ possible ways to choose $\tau$.
    Each such choice of $\tau$ yields a distinct $k$-simplex, since the intersection of two  distinct $(k+1)$-simplices can not contain two distinct $k$-simplices.
    Therefore, starting from the initial $k$-simplex $\sigma$, we can find $n-k$ additional $k$-simplices, leading to a total of at least $n-k+1$ such simplices.
\end{proof}

A non-empty orthant $s$ for which this inequality holds as equality is called \textbf{simplicial}.
The name come from the fact that for $s \in \mathcal E(\Delta_n)$, the $n-k$-cell $C_{\mathcal E}(s(\Delta_n))$ contains a number of 0-cells corresponding to $|\{ \sigma \in F_k(\Delta)| s \in \mathcal E(\sigma) \}|$.
Thus, if $s \in \mathcal E(\Delta_n)$ is simplicial, the cell $C_{\mathcal E}(s(\Delta_n))$ is a simplicial cell in the canonical cell decomposition of $X_\mathcal E$.
\begin{Ex}
    \begin{itemize}
        \item If $k=n$, the polytope $\Delta_n$ has only one $k$-simplex (which is $\Delta_n$ itself), thus the unique non-empty orthant of $\mathcal E$ is simplicial.
        \item If $k=0$, the $k$-simplices are the vertices of $\Delta_n$ and $\mathcal E(v) = \mathcal Q^n$. Thus, all orthants are simplicial.
        \item If $k=n-1$, all non-empty orthants are simplicial. Indeed, $C_{\mathcal E}(s(\Delta_n))$ is a cell of dimension $1$ and thus contains exactly two 0-cells.
        \item In the Example \ref{ex2}, the orthant $(-,+,-)$ is empty.
        The four orthants $(-,+,+)$, $(+,-,+)$,$(+,+,-)$ and $(-,-,-)$, represented in green by in Figure \ref{n3k1}, are simplicial.
        The three other orthants $(+,+,+)$, $(-,-,+)$ and $(+,-,-)$, represented in red, are non-simplicial.
    \end{itemize}
\end{Ex}

\begin{lemma}\label{l:gamma}
    Suppose $k \leq n-1$.
    Let $s \in \mathcal E(\Delta_n)$ be a simplicial orthant. Then there exists a $(k-1)$-simplex $\gamma$ such that for all $k$-simplex $\sigma$,
    \[ s \in \mathcal E(\sigma) \iff \gamma \subset \sigma. \]
\end{lemma}
    \begin{proof}
        First, notice that by the Lemma \ref{minorthant}, if $s$ is simplicial then all $k$-simplices $\sigma$ such that $s \in \mathcal E(\sigma)$ are adjacent to each other, and no three of them can lie in a same $(k+1)$-simplex.
        
        Fix three of these $k$-simplices: $\sigma^{(0)}$, $\sigma^{(1)}$ and $\sigma^{(2)}$. We only need to show that $\gamma^{(1)} = \sigma^{(1)} \cap \sigma^{(0)}$ and $\gamma^{(2)} = \sigma^{(2)} \cap \sigma^{(0)}$ are the same $(k-1)$-simplex.
        
        Since $\sigma^{(1)}$ and $\sigma^{(2)}$ are adjacent, they are both in a $(k+1)$-simplex $\tau$.
        The simplex $\tau \cap \sigma^{(0)}$ is contained in the $k$-simplex $\sigma^{(0)}$, and it contains the $(k-1)$-simplex $\sigma^{(1)} \cap \sigma^{(0)} = \gamma^{(1)}$.
        
        Thus, we have two possibilities: $\tau \cap \sigma^{(0)} = \sigma^{(0)}$ or $\tau \cap \sigma^{(0)} = \gamma^{(1)}$.
        
        If we are in the first case, then $\sigma^{(0)}$, $\sigma^{(1)}$ and $\sigma^{(2)}$ are both contained in the $(k+1)$-simplex $\tau$, which contradicts that $s$ is simplicial.
        Thus, we must be in the second case: $\tau \cap \sigma^{(0)} = \gamma^{(1)}$.

        Since by construction $\sigma^{(2)} \subset \tau$, we also have $\gamma^{(2)} = \sigma^{(2)} \cap \sigma^{(0)} \subset \tau \cap \sigma^{(0)} = \gamma^{(1)}$.
        By equality of the dimensions, we can conclude that $\gamma^{(1)} = \gamma^{(2)}$.
    \end{proof}

\begin{prop}\label{nbsimplicial}
    The number $S_{n,k}$ of simplicial orthants in $\mathcal E$ is at most $\frac{2}{k+1} \binom{n+1}{k}$.
\end{prop}

\begin{proof}
    We proceed by induction on $n$.
    If $n=k$, there is only one orthant which is always simplicial. Thus, we have $S_{n,n} = 1 < \frac{2}{n+1}\binom{n+1}{n}$.

    If $n=k+1$, the bound is also trivial:
    \[ S_{n,n-1} = |\mathcal E(\Delta)| = n+1 = \frac{2}{n}\binom{n+1}{n-1} \]
    
    Suppose $k<n-1$.
    Let $s$ be simplicial orthant and $\gamma$ the $(k-1)$-simplex described in Lemma \ref{l:gamma}.
    For each face $F$ of $\Delta_n$ containing $\gamma$, observe that $\pi_F(s)$ is a simplicial orthant of $\mathcal E_{|F}$.
    Indeed, for every $k$-simplex $\sigma$ of $F$, we have the equivalence 
    \[ \pi_F(s) \in \mathcal E_{|F}(\sigma) \Longleftrightarrow s \in \mathcal E(\sigma).\]
    Among all $k$-simplices of $\Delta_n$ containing $\gamma$, exactly one does not lie in $F$, therefore
    $|\{ \sigma \in F_k(F)| s \in \mathcal E_{|F}(\sigma) \}| = |\{ \sigma \in F_k(\Delta_n)| s \in \mathcal E(\sigma) \}| -1 = n-k$.
    Hence, $\pi_F(s)$ is indeed simplicial in $\mathcal E_{|F}$.

    Consequently, a simplicial orthant $s$ in $\mathcal E$ gives a simplicial orthant in $\mathcal E_{|F}$ for each of the $n+1-k$ facets $F$ containing $\gamma$.

    By a double counting of the pairs $(s, F)$ where $s$ is a simplicial orthant  of $\mathcal E$ such that $\pi_F(s)$ is a simplicial orthant of $\mathcal E_{|F}$, we obtain
    $(n+1-k)S_{n,k} \leq (n+1) S_{n-1,k}$.
    By induction, we have $S_{n,k} \leq  \frac{n+1}{n+1-k} \frac{2}{k+1} \binom{n}{k} = \frac{2}{k+1} \binom{n+1}{k}$.
    \end{proof}

    \begin{proof}[Proof of Theorem \ref{nbsimplex}]
        For each $n$-simplex $\sigma$ of $\Gamma$,
        the simplex $s(\sigma)$ contains a simplicial cell of $X_{\mathcal{E}}$ if and only if $s$ is a simplicial orthant of $\mathcal E(\sigma)$.
        Thus, by Theorem \ref{nbsimplicial}, each $n$-simplex of $\gamma$ contributes to at most $\frac{2}{k+1} \binom{n+1}{k}$ simplicial orthants.
        Since $\Gamma$ is unimodular, it contains a number of $n$-simplices equal to $\Vol(\Delta)$.
    
        Hence, the total number of simplicial cells in the canonical decomposition of $X_\mathcal E$ is at $\frac{2}{k+1} \binom{n+1}{k} \Vol(\Delta)$.
    \end{proof}

\section{Bounds on the number of connected components of T-manifolds}

\subsection{Maximal T-curves}

In this section, we consider a lattice polytope $\Delta$  in $\mathbb R^n$ equipped with a unimodular triangulation $\Gamma$ and  a real phase structure $\mathcal E$ on the $(n-1)$-skeleton of $\Gamma$.
In this context, the T-manifold $X_\mathcal E$ is a curve. The canonical cell decomposition of $X_\mathcal E$ consists of 0-cells and 1-cells:
 \begin{itemize}
    \item The 0-cells $C_{\mathcal E}(s(\sigma))$ correspond to the $(n-1)$-simplices $\sigma$ of $F_{n-1}(\Gamma)$ with $s \in \mathcal E(\sigma)$.
    \item The 1-cells $C_{\mathcal E}(s(\tau))$ correspond to the $n$-simplices $\tau$ of $F_n(\Gamma)$ with $s \in \mathcal E(\tau)$ 
 \end{itemize}

We will study the number of connected components $b_0(X_{\mathcal E})$ of the T-curve $X_\mathcal E$.

 \begin{defi}\label{dualgraph}
     The dual graph $G_\Gamma$ of the subdivision $\Gamma$ is defined as follows:

     The set of vertices of $G_\Gamma$ consists of $v_\tau$ for each $n$-simplex $\tau \in F_n(\Gamma)$ and $u_\sigma$ for each $(n-1)$-dimensional simplex $\sigma \in F_n(\Gamma)$ that lies on the boundary of $\Delta$.
     
     The set of edges of $G_\Gamma$ consist of $e_\sigma$ for each $(n-1)$-simplex $\sigma$ of $F_{n-1}(\Gamma)$.
     
     \begin{itemize}
        \item If $\sigma$ is an interior $(n-1)$-simplex, the vertices adjacent to the edge $e_\sigma$ correspond to the two adjacent $n$-simplices: $v_{\tau}$ and $v_{\tau'}$.
        \item If $\sigma$ is a boundary $(n-1)$-simplex, the vertices adjacent to the edge $e_\sigma$ correspond to the vertex $v_\tau$ of the unique adjacent $n$-simplex and the vertex $u_\sigma$.
     \end{itemize}
 \end{defi}

 \begin{figure}
    \centering
   \includegraphics{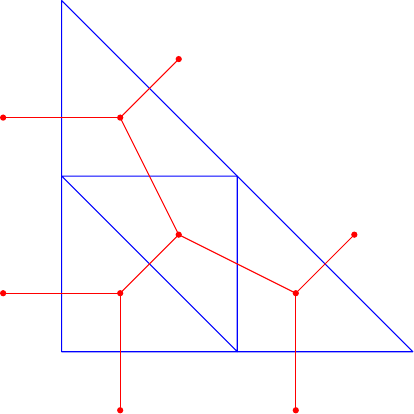}
   \caption{In blue, a subdivision and, in red, its dual graph.}
 \end{figure}

\begin{Rem}
    This kind of graph appears in tropical geometry.
    To a tropical hypersurface $\Sigma$ in $\mathbb R^n$ is associated a dual polytope known as the Newton polytope. The tropical hypersurface also induces a convex triangulation $\Gamma$
of the Newton polytope. In this case, the dual graph $G_\Gamma$ is the graph of the 1-skeleton of $\Sigma$ where the unbounded rays of $\Sigma$ have been compactified by vertices.

\end{Rem}

 For a triangulation $\Gamma$ of $\Delta \subset \mathbb R^n$, the vertices $v_\tau$  (corresponding to the $n$-simplices) of $G_\Gamma$ have degree $n+1$,
 while the vertices $u_\sigma$ (corresponding to the $(n-1)$-simplices on the boundary) have degree 1.

For a smooth real algebraic compact curve, by the Harnack-Klein theorem \cite{Klein}, the number of connected components of the real part is bounded by the genus of the complex part plus one.
As observed by Haas in \cite{Haas}, this inequality has a combinatorial analogue for T-curves in toric surfaces.
The following proposition generalizes this inequality to higher ambient dimensions and gives a condition for the equality to hold.
In this setting, the role of the genus of the complex part is played by the first Betti number $b_1(G_\Gamma)$ of the graph homology.

Recall that a graph $G$ is said planar if it can be embedded (with no crossing edges) in $\mathbb R^2$ or equivalently in the sphere $S^2$.

\begin{prop}\label{boundcc}
    Let $(\Delta, \Gamma)$ be a triangulated polytope of dimension $n$ and $\mathcal E$ be a real phase structure on the $(n-1)$-skeleton of $\Gamma$.
    We have the following bound:
    \begin{equation}\label{eq:boundcc}
        b_0(X_{\mathcal E}) \leq b_1(G_\Gamma)+1
    \end{equation}
    Moreover in the case of equality, the graph $G_\Gamma$ is planar.
\end{prop}

\begin{proof}
    First, observe that to each 0-cell $C_{\mathcal E}(s(\sigma))$, with $\sigma \in F_{n-1}(\Gamma)$ and $s \in \mathcal E(\sigma)$, we can associate the edge $e_\sigma$ in $G_\Gamma$.
    Conversely, since for $\sigma \in F_{n-1}(\Gamma)$ the affine space $\mathcal E(\sigma)$ is one dimensional, it consists of exactly two orthants $\{s, s'\}$.
    Thus, the edge $e_\sigma$ of $G_\Gamma$ is associated to two 0-cells $C_\mathcal E(s(\sigma))$ and $C_\mathcal E(s'(\sigma))$ in $X_\mathcal E$.
    If the two $(n-1)$-simplices $s(\sigma)$ and $s(\sigma')$ are identified in $\widetilde{\Delta}$ (which happen exactly when $\sigma$ is a boundary simplex),
    $C_\mathcal E(s(\sigma))$ and $C_\mathcal E(s'(\sigma))$ are in fact the same cell but, in this case, we count the cell twice.

    For each connected component $A$ of $X_\mathcal E$, the succession of 0-cells of $A$ defines a succession of edges in $G_\Gamma$,
    which in turn forms a cycle $\gamma_A$ in $G_\Gamma$ (potentially with repeating edges).
    
    For each cycle $\gamma_A$, attach a disk by gluing their boundary along $\gamma_A$. This defines a two-dimensional compact CW-complex $S$, which we claim is a topological surface.
    Indeed, since each edge correspond to two 0-cells, it appears either in two cycles $\gamma_A$ or twice in a single cycle. This implies that the interior of an edge locally looks like the gluing of two half-planes.
    Furthermore, for a $n$-simplex $\tau$, by Lemma \ref{necklace},
    we have a cyclic order $\sigma_0, \sigma_1, \dots, \sigma_n$ of the $(n-1)$-simplices in $\tau$ and orthant $s_i \in \mathcal Q^n$ such that $\mathcal E(\sigma^{(i)}) = \{s_i, s_{i+1}\}$ taking the value of $i$ modulo $n+1$.
    Thus, there are cycles $\gamma_{A_i}$ containing successively $e_{\sigma^{(i)}}$ and $e_{\sigma^{(i+1)}}$.
    Hence, the neighborhood of $v_\tau$ in $S$ is homeomorphic to $\mathbb R ^2$ (see Figure \ref{gluingS}).
    Finally, for a $(n-1)$-simplex $\sigma$ on the boundary with $\mathcal E(\sigma)=\{s_1, s_2\}$, the 1-cells $C_{\mathcal E}(s_1(\sigma))$ and $C_{\mathcal E}(s_2(\sigma))$ are adjacent and a fortiori belong to the same connected component $A$.
    Thus, the neighborhood of $u_\sigma$ is also homeomorphic to $\mathbb R^2$.
    Therefore, $S$ is indeed a closed connected surface.

    Let's compute its Euler characteristic $\chi(S)$:
    the complex $S$ is composed of $V$ vertices, $E$ edges and $F=b_0(X_{\mathcal E})$ faces,
    so $\chi(S) = V - E + F$.
    But looking at Euler characteristic of $G_\Gamma$, we have
    $V - E = \chi(G_\Gamma) = b_0(G_\Gamma) - b_1(G_\Gamma) = 1- b_1(G_\Gamma)$ since $G_\Gamma$ is connected.
    So $\chi(S) = b_0(X_{\mathcal E}) + 1 -b_1(G_\Gamma)$. But for a connected compact surface, the Euler characteristic is less than 2 with the equality case only for the sphere. 
    It follows that $b_0(X_{\mathcal E}) + 1 -b_1(G_\Gamma) \leq 2$ or  $b_0(X_{\mathcal E}) \leq b_1(G_\Gamma)+1$
    with equality only if $S$ is a sphere.
    Furthermore, in this case, the inclusion $G_\Gamma \subset S$ gives a planar embedding of $G_\Gamma$.
\end{proof}

\begin{figure}
    \includegraphics{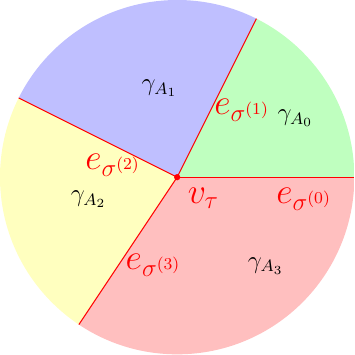}
    \includegraphics{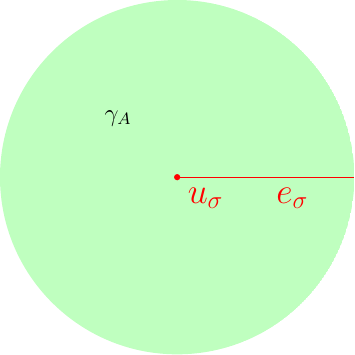}

\caption{An illustration of the gluing defining the surface $S$ in the proof of proposition \ref{boundcc}}
\label{gluingS}
\end{figure}

\begin{Rem}
    When $\Delta$ is smooth, it follows from the Theorem 1.4 of \cite{brugalle2024combinatorial} that the inequality \ref{boundcc}
    correspond to the case $k=n-1$ and $p=0$ of the inequality \ref{boundhodge}.
    More precisely, $h^{0,1}(Y_\Delta^{n-1}) = b_1(G_\Gamma)$ and $h^{0,0}(Y_\Delta^{n-1}) = 1$.
    In this light, we say that the T-curve $X_\mathcal E$ is maximal when $b_0(X_{\mathcal E}) = b_1(G_\Gamma) + 1$.

    The proof above provides a direct and elementary justification of this special case, which additionally gives a condition on the equality case.
\end{Rem}

When the dual graph is trivalent, a study of maximal T-curve can be found in \cite{Hassrevis}.

By Kuratowski's theorem \cite{KuratowskiThm}, a graph is planar if and only if it contains no
subgraph that is a subdivision of the complete graph $K_5$ or the complete bipartite graph $K_{3,3}$.

\begin{prop}\label{Kn}
    Let $\Delta$ be an $n$-dimensional lattice polytope with a unimodular triangulation $\Gamma$.
    If $\Delta$ has an interior lattice point $p$, then $G_\Gamma$ contains a subgraph that is a subdivision of the complete graph on $n+1$ vertices $K_{n+1}$.
\end{prop}

\begin{proof}
    Since $\Gamma$ is unimodular, the point $p$ is the vertex of an $n$-dimensional simplex $\sigma$ of $\Gamma$.
    Each facet of $\sigma$ containing $p$ is an interior simplex of $\Delta$ and thus is contained in another maximal simplex.
    Let $\sigma_1, \dots \sigma_n$ denote these maximal simplices.
    Consider the vertices $v, v_1, \dots v_n$ of the graph $G_\Gamma$ corresponding to the $n$-simplex $\sigma, \sigma_1, \dots, \sigma_n$.
    
    We claim that every pair of these $n+1$ vertices can be connected by a path in $G_\Gamma$ such that these paths are disjoint (excepted for their endpoints).
    Together, they form a subdivision of $K_{n+1}$.

    The path between $v$ and $v_i$ is just the edge $v-v_i$, since $\sigma$ and $\sigma_i$ share a facet.

    For $v_i$ and $v_j$, consider the intersection $\sigma \cap \sigma_i \cap \sigma_j$: this is a $(n-2)$-simplex $\gamma_{i,j}$ containing $p$.
    As $\gamma_{i,j}$ is an interior $(n-2)$-simplex, the vertices of $G_\Gamma$ corresponding to simplex containing $\gamma_{i,j}$ form a cycle.
    This cycle includes the path $v_i-v-v_j$ and the complementary part of the cycle provides the desired path between $v_i$ and $v_j$.

    We now show that the paths are disjoints. It is direct for the path $v-v_i$.
    Suppose $w$ is a vertex contained both in the path between $v_i$ and $v_j$ and in the path between $v_{i'}$ and $v_{j'}$, with $\{i,j\} \neq \{i',j'\}$.
    The vertices $w$ of $G_\Gamma$ must correspond to an $n$-simplex of $\Gamma$ containing both $\gamma_{i,j}$ and $\gamma_{i',j'}$.
    Since the simplex containing $\gamma_{i,j}$ and $\gamma_{i',j'}$ is a facet of $\sigma$ containing $p$ or $\sigma$ itself, the vertex $w$ can only be $v, v_1, \dots v_{n-1}$ or $v_n$.
    Hence, all our paths only meet at their endpoints.
\end{proof}

We immediately deduce the following restriction on the existence of maximal T-curves.

\begin{coro}\label{nomaxTcurve}
    Let $\Delta$ be a lattice polytope of dimension $n>3$ with an interior lattice point.
    For any real phase structure $\mathcal E$ on the $(n-1)$-skeleton of a triangulation of $\Delta$,
    the T-curve $X_\mathcal E$ is not maximal.
\end{coro}
\begin{proof}
    By proposition \ref{Kn}, the graph $G_\Gamma$ contains a subdivision of $K_{n+1}$ with $n \geq 4$.
    Since $K_m$ is not planar for $m \geq 5$, the graph $G_\Gamma$ is not planar.
    Hence, by proposition \ref{boundcc}, the T-curve $X_\mathcal E$ is not maximal.
\end{proof}

Ehrhart theory of lattice polytope provides powerful tools to study the existence of interior points.
See \cite{Beck} for details on this topic.

\begin{defi}
    The codegree of a lattice polytope $\Delta$ is the minimal integer $k$ such that
    $k\Delta$ has no  interior lattice point.
\end{defi}

The codegree can be computed from the $f$-vector of a unimodular triangulation $\Gamma$.
But we restrict ourselves to this simple bound on the codegree.

\begin{prop}\label{codeg}
    The codegree of a lattice polytope $\Delta$ of dimension $n$ is less than $n+1$
    with equality if and only if $\Delta$ is a unimodular simplex.
\end{prop}

\begin{coro}
    Let $\Delta$ be a lattice polytope of dimension $n>3$ and $d \geq n+1$.
    For any real phase structure $\mathcal E$ on the $(n-1)$-skeleton of a triangulation of $d\Delta$,
    the T-curve $X_\mathcal E$ is not maximal.
    If $\Delta$ is not a unimodular simplex, we can even take $d=n$.
\end{coro}
\begin{proof}
    By proposition \ref{codeg}, the hypothesis of Corollary \ref{nomaxTcurve} is satisfied hence also its conclusion.
\end{proof}

For $n=2$, maximal T-curve can be constructed for any polygon\cite{Haas}.

For $n=3$, the assumption that $G_\Gamma$ is planar is already pretty restrictive.
However, we present in Section \ref{sec:maxcurve} a maximal T-curve in $d\Delta_3$, for any $d \geq 1$.

\subsection{Bounds on the number of connected components of T-curves}

In the previous subsection, we have shown there are no maximal T-curves in ambient dimension $n>3$ when the polytope contains interior points.
But can we still have T-curve approaching maximality? Is the bound of proposition \ref{boundcc} still, at least asymptotically, optimal?

Here, we construct a new bound on the number of connected component of T-curves by enumerating the number of 1-cells of a T-curve.
It turns out this bound is stronger than the one from \ref{boundcc} for dilation of polytope in high dimension.

\begin{theo}
    Let $\Delta$ be a lattice polytope of dimension $n$. For any T-curve $X_{\mathcal E}$ with  Newton polytope $\Delta$, one has
    \begin{equation}\label{newbound}
        b_0(X_\mathcal E) \leq \frac{n+1}{3} \Vol(\Delta)
    \end{equation}
\end{theo}

\begin{proof}
By Theorem \ref{nbfacet}, the number of 1-cells in the canonical cell decomposition of $X_\mathcal E$ is $(n+1)\Vol(\Delta)$.

However, because two maximal simplices share at most one facet, a connected component of $X_{\mathcal E}$ must pass through at least three different maximal simplices of $\widetilde {\Delta}$.
Thus, each of the $b_0(X_{\mathcal E})$ connected components of $X_\mathcal E$ is made of at least three 1-cells. Hence, the inequality:
\[ 3b_0(X_{\mathcal E}) \leq (n+1)\Vol(\Delta) \]
\end{proof}

To compare the bounds \ref{eq:boundcc} and \ref{newbound}, we need to relate the homology of $G_\Gamma$ with the volume of $\Delta$.

By a slight abuse of notation, we denote $\Vol(\partial \Delta)$ the sum, over all facet $F$ of $\Delta$, of the lattice volume $\Vol(F)$, where $F$ is considered as a full dimensional polytope in their tangent lattice $T(F)$.
For example, we have $\Vol(\partial \Delta_n) = n+1$.

\begin{lemma}\label{bettigraph}
    Let $\Delta$ be an $n$-dimensional polytope with $\alpha = \Vol(\Delta)$ and $\beta = \Vol(\partial \Delta)$.
    Then for any unimodular triangulation $\Gamma$, the first Betti number of $G_\Gamma$ is $b_1(G_\Gamma) = \frac{(n-1)\alpha-\beta}{2} +1$.
\end{lemma}

\begin{proof}
The unimodular triangulation $\Gamma$ contains $\alpha$ maximal simplices and $\beta$ $(n-1)$-simplices on the boundary of $\Delta$.
Thus, $G_\Gamma$  is made of $\alpha$ vertices of degree $n+1$ and $\beta$ vertices of degree 1.
By the handshaking lemma, the number of edges of $G_\Gamma$ is thus $\frac{(n+1)\alpha+\beta}{2}$.
Computing Euler characteristic leads to: $b_0(G_\Gamma) - b_1(G_\Gamma) = (\alpha+\beta) - \frac{(n+1)\alpha+\beta}{2}$.
Given that $G_\Gamma$ is connected, we obtain $b_1(G_\Gamma) = \frac{(n-1)\alpha-\beta}{2} +1$.
\end{proof}

Considering dilation of $\Delta$, we have the following volume relationships: 
\[\Vol(d\Delta) = d^n\Vol(\Delta)=d^n\alpha \text{ and } \Vol(\partial(d\Delta))= d^{n-1} \Vol(\partial \Delta)=d^{n-1}\beta.\]
Thus, for a T-curve $X_d$  with  Newton polytope $d\Delta$,
the proposition \ref{boundcc} gives us
 \begin{equation}\label{dbound}
    b_0(X_d) \leq  \frac{(n-1)d^n\alpha-d^{n-1}\beta}{2}+2
 \end{equation} with equality for maximal curves.
 Meanwhile, the bound (\ref{newbound}) is $b_0(X_d) \leq \frac{(n+1)d^n\alpha}{3}$.
For $n>5$, the latter bound is asymptotically better as $d$ tends to infinity.

\begin{prop}
    Let $\Delta$ be a lattice polytope of dimension $n>5$ with volume $\alpha =\Vol(\Delta)$ and boundary volume $\beta = \Vol(\partial \Delta)$.
    For a T-curve $X_d$ in $d\Delta$, the bound (\ref{newbound}) is stricter than the bound (\ref{eq:boundcc}), for $d$ large enough.
    More precisely, \[b_0(X_d) \leq \frac{n+1}{3}d^n \alpha < b_1(G_\Gamma)+1\]
    for any $d>2$ such that $(n-5)\alpha d \geq 3\beta$.
\end{prop}

\begin{proof}
    By equation (\ref{inequality1}), we have $b_1(G_\Gamma)+1 = \frac{(n-1)d^n\alpha-d^{n-1}\beta}{2}+2$
    Thus, the inequality to check is
    \[ \frac{n+1}{3}d^n \alpha < \frac{(n-1)d^n\alpha-d^{n-1}\beta}{2}+2 \] which is equivalent to
    \begin{equation}\label{inequality1}
        3\beta < (n-5)d\alpha +12d^{1-n}.
    \end{equation}
    Since we supposed $n>5$ and $d>1$, we have $12d^{1-n}<\frac{3}{8} <1$.
    Hence, (\ref{inequality1}) is equivalent to $ 3\beta \leq (n-5)d\alpha -1$ since the quantities are integers.
\end{proof}

\begin{coro}
    Let $\Delta$ be a lattice polytope of dimension $n>5$.
    There is no asymptotically maximal family $(X_d)_{d \geq 1}$ of T-curve with  Newton polytope $\Delta$.
\end{coro}

 \subsection{Bounds on the number of connected components of T-surfaces}

 We now use a slight variation of the method used for T-curves to derive new bounds on the number of connected components of maximal T-surfaces.
 In this subsection $\mathcal E$ is $(n-2)$-real phase structure.
 
 \begin{theo}\label{boundsurface}
     Let $\Delta$ be an $n$-dimensional lattice polytope and let $X_\mathcal E$ be a T-surface in $\Delta$. Then,
     \[ b_0(X_\mathcal E) \leq \frac{7n^2+5n+12}{60}\Vol(\Delta).\]
 \end{theo}
 \begin{proof}
     By Theorem \ref{nbfacet}, $X_\mathcal E$ consists of $f =[\frac{n(n-1)}{2}+n+1]\Vol(\Delta)$ faces, and by Theorem \ref{nbsimplex}, at most $t =\frac{2}{n-1}\binom{n+1}{n-2}\Vol(\Delta)=\frac{n(n-1)}{3}\Vol(\Delta)$ of these are triangles.
     
      Let $p_m$ be the number of connected components of $X_\mathcal E$ with exactly $m$ 2-faces.
      The total number of connected components is $b_0(X_\mathcal E) = \sum_{m} p_m$.
      Since a face of $X_\mathcal E$ has at least three adjacent faces, a connected component of $X_\mathcal E$ is made of at least 4 faces, thus $p_i=0$ for $i <4$. 
      In a component with exactly four faces (that is contributing to $p_4$), all these faces must be triangles.
      Thus, we obtain $4p_4 \leq t$.
      The total number of faces can be estimated as follows:
      $f = \sum_{m \geq 0} m p_m = 4p_4 + \sum_{m \geq 5}mp_m \geq 4p_4 + \sum_{m \geq 5} 5p_m = 5b_0(X_\mathcal E)-p_4$.
      Combining these inequalities yields
      \[b_0(X_\mathcal E) \leq \frac{4f+t}{20} = \frac{7n^2+5n+12}{60}\Vol(\Delta).\]
    \end{proof}
 
 To compare this bound to the maximality criterion (\ref{boundhodge}), we need to evaluate the Hodge numbers $h^{0,q}(Y_\Delta^k)$.
 Khovanskii \cite{Khovanskii} computed these numbers for any generic complete intersection in toric variety:
 
 \begin{align*}
 h^{0,q}(Y^k_\Delta) &= 
   \begin{cases}
     1 & \text{if } q=0\\
     \sum_{l=1}^k \binom{k}{l}(-1)^{k-l}|\text{int}(l\Delta) \cap \mathbb Z^n|\ & \text{if } q=n-k\\
     0 & \text{otherwise}
   \end{cases}
 \end{align*}
 
 This  expression can be rewritten in terms of the Ehrhart polynomial using the reciprocity law:
 \[h^{0,n-k}(Y_\Delta^k)= \sum_{l=1}^k\binom{k}{l}(-1)^{n+k-l}P_\Delta(-l),\]
 where $P_\Delta$ is the Ehrhart polynomial of $\Delta$, defined by $P_\Delta(t) = |t\Delta \cap \mathbb Z^n|$, for $t$ a positive integer.
 
 \begin{lemma}
     The sum $\sum_{q=0}^{n-k}h^{0,q}(Y_{d\Delta}^k)$ is a rational polynomial in $d$ whose leading term is
     \[\frac{k!}{n!} {n \brace k} \Vol(\Delta)d^n,\]
     where  ${n \brace k}$ denotes the Stirling number of the second kind.
 \end{lemma}
 \begin{proof}
     Since $P_{d\Delta}(x) =P_\Delta(dx)$, we have
     \[ \sum_{q=0}^{n-k}h^{0,q}(Y_{d\Delta}^k) = 1 + \sum_{l=1}^k \binom{k}{l}(-1)^{k-l}P_\Delta(-dl).\]
     As $P_\Delta$ is a rational polynomial of degree $n = \dim \Delta$ with leading coefficient $\frac{\Vol(\Delta)}{n!}$, the sum of Hodge numbers is a polynomial in $d$ with leading term
     \[ \sum_{l=1}^k (-1)^{n+k-l}\binom{k}{l} \frac{\Vol(\Delta)}{n!}(-ld)^n.\]
     The formula for Stirling numbers of the second kind ${n \brace k} = \frac{1}{k!} \sum_{l=0}^k (-1)^{k-l}\binom{k}{l}l ^n$ gives the desired leading term $\frac{k!}{n!} {n \brace k} \Vol(\Delta)d^n$.
 \end{proof}
 
 For surfaces, we have $n-k=2$. Hence, the leading term is
  \[\frac{(n-2)!}{n!}{n \brace n-2} \Vol(\Delta)d^n = \frac{(n-2)(3n-5)}{24}d^n \Vol(\Delta)\]

 Therefore, an asymptotically maximal family of T-surfaces $(X_d)_{d \geq 1}$ with polytope $\Delta$ must have the following asymptotic number of connected components, as $d$ tends to infinity:
  $b_0(X_d) \sim \frac{(3n-5)(n-2)}{24}d^n \Vol(\Delta)$.
 
  \begin{coro}
      There is no family of asymptotically maximal T-surface in $\Delta$ for $n=\dim  \Delta \geq 65$.
  \end{coro}
 
  \begin{proof}
     Let $X_d$ be a T-surface in $d\Delta$. By Theorem \ref{boundsurface},
      \[b_0(X_d) \leq \frac{7n^2+5n+12}{60}d^n\Vol(\Delta).\]
      Thus, for a family of T-surfaces $(X_d)_{d \geq 1}$ to be asymptotically maximal,
      we must have $\frac{(3n-5)(n-2)}{24} \leq \frac{7n^2+5n+12}{60}$.
      Reorganizing, we obtain the quadratic inequality $n^2-65n+26 \leq 0$, which is not satisfied for integers greater than or equal to 65. 
      
  \end{proof}
 
\section{Construction of $k$-real phase structures}\label{stableintersection}

Except in the case of codimensions 0, 1 and $n$, the description of real phase structure is rather intricate.
Even the existence of a real phase structure on the $k$-skeleton of a unimodular triangulation $\Gamma$ is not obvious at first sight.

In this section, we explain a construction that produces real phase structures of higher codimension from two real phase structures of lower codimension.

One way to interpret this construction is as a real analogue of stable intersection in tropical geometry.
Two tropical varieties $\Sigma_1$ and $\Sigma_2$ in $\mathbb R^n$ may not intersect transversely.
However, the Hausdorff limit of $\Sigma_1 \cap (tv+\Sigma_2)$, as $t$ goes to 0, defines a tropical variety, independent of the choice of a generic translation vector $v$,
called the stable intersection of $\Sigma_1$ and $\Sigma_2$ \cite[Section 4.3]{ TropicalBook}.

Likewise, the classical intersection of two T-manifolds is typically not a manifold. But, locally, if one slightly translates a T-manifold relatively from another, the result will generically be a transverse intersection,
but, this time, it is dependent of the direction of translation.
In this section, we introduce a combinatorial construction formalizing this idea of stable intersection at the level of real phase structures.

\subsection{Stable intersections of real phase structures}

Fix a polytope $\Delta$ and a unimodular triangulation $\Gamma$.

\begin{defi}
    A locally acyclic orientation $\mathcal O$ on $F_1(\Gamma)$ is an orientation of the graph of edges of $\Gamma$ such that no simplex of $\Gamma$ contains an oriented cycle.
\end{defi}

For any simplex $\sigma \in \Gamma$, a locally acyclic orientation $\mathcal O$ induces an ordering on the vertices of $\sigma$ by the following rule:
for two distinct vertices $u$ and $v$ of $\sigma$, $u<v$ if and only if the edge $[u,v]$ is oriented from $v$ to $u$.
This order is well-defined by acyclicity, and it is unique since the graph induce by the simplex $\sigma$ is complete.
From now on, if $\Gamma$ is a triangulation endowed with a locally acylcic orientation,
the notation $\sigma = [v_0, \dots v_k]$ for a simplex of $\Gamma$ also implies that $v_0 < \dots <v_n$.

 Let $\mathcal E_1, \mathcal E_2$ be real phase structures on the $k_1$ and $k_2$-skeleton of $\Gamma$ with $k_1+k_2 \leq n$, and let $\mathcal O$ be a locally acyclic orientation of $\Gamma$.
 We define the $(k_1+k_2)$-real phase structure $\mathcal E$ as follows.
 Let $\sigma= [v_0, \dots, v_{k_1+k_2}]$ be a $(k_1+k_2)$-simplex of $\Gamma$.
 Set the $k_1$-simplex $\sigma_1 =[v_0, v_1, \dots, v_{k_1}]$ and the $k_2$-simplex $\sigma_2 = [v_{k_1}, \dots v_{k_1+k_2}]$.
We define:
\[ (\mathcal E_1 \cap_{\mathcal O} \mathcal E_2)(\sigma) = \mathcal E_1 (\sigma_1) \cap \mathcal E_2(\sigma_2) \]

 \begin{lemma}\label{intersectionrps}
   The intersection structure $ \mathcal E = \mathcal E_1 \cap_{\mathcal O} \mathcal E_2$ defines a $(k_1+k_2)$-real phase structure.
   Moreover, the T-manifold $X_{\mathcal E}$ is contained in the intersection $X_{\mathcal E_1} \cap X_{\mathcal E_2}$.
 \end{lemma}
 \begin{proof}
 We keep the notation introduced above.
 Since $\Gamma$ (and hence $\sigma$) is unimodular, $T(\sigma) = T(\sigma_1) \oplus T(\sigma_2)$. Reducing modulo 2 and dualizing, we obtain that $T_2^\perp(\sigma) = T_2^\perp(\sigma_1) \cap T_2^\perp(\sigma_2)$. Since $T_2^\perp(\sigma_1) + T_2^\perp(\sigma_2) = \mathbb F_2^n$, the intersection of affine spaces $\mathcal E(\sigma) = \mathcal E_1 (\sigma_1) \cap \mathcal E_2(\sigma_2)$ can not be empty, thus it is an affine space directed by $T_2^\perp(\sigma)$, as required.

 To verify the parity condition, fix $s \in \mathcal Q^n$, and let $\tau =[v_0, v_1, \dots, v_{k_1+k_2+1}]$ be a $(k_1+k_2+1)$-simplex of $\Gamma$.
 For each $0 \leq i \leq k_1+k_2+1$, let $\sigma_i = [v_0, \dots, \hat{v_i}, \dots, v_{k_1+k_2+1}]$ be the $(k_1+k_2)$-simplex opposite to $v_i$ in $\tau$. 
 
 By definition of $\mathcal E$, for $0\le i\le k_1$,
\[
\mathcal E(\sigma_i)=\mathcal E_1(\eta_i)\cap \mathcal E_2(\gamma_2),
\quad
\eta_i=[v_0,\ldots,\widehat{v_i},\ldots,v_{k_1+1}],
\gamma_2=[v_{k_1+1},\ldots,v_{k_1+k_2+1}],
\]
and for $k_1+1\le i\le k_1+k_2+1$,
\[
\mathcal E(\sigma_i)=\mathcal E_1(\gamma_1)\cap \mathcal E_2(\eta_i),
\quad
\gamma_1=[v_0,\ldots,v_{k_1}],
\eta_i=[v_{k_1},\ldots,\widehat{v_i},\ldots,v_{k_1+k_2+1}].
\]
 
 By the parity condition of $\mathcal E_1$ with the $(k_1+1)$-simplex $[v_0, \dots, v_{k_1+1}]$, there is an even number $n_1$ of indices $i$ between 0 and $k_1$ such that $s \in \mathcal E_1(\eta_i)$. Similarly, there is an even number $n_2$ of indices $i$ between $k_1+1$ and $k_1+k_2+1$ such that $s \in \mathcal E_2(\eta_i)$.
 So the number of indices $i$ such that $s \in \mathcal E(\sigma_i)$ is $\delta_1 n_1 +\delta_2 n_2$, where $\delta_j$ is 1 if $ s \in \mathcal E_j(\gamma_j)$ and 0 otherwise, for $j=1,2$.
 Since this number is even, the parity condition on $\mathcal E$ is verified.

 Because of the inclusion $\mathcal E(\sigma) \subset \mathcal E_i(\sigma)$, each cell $C_\mathcal E(s(\sigma))$ is contained in the cell $C_{\mathcal E_i}(s(\sigma))$. This proves the inclusion $X_\mathcal  E \subset X_{\mathcal E_1} \cap X_{\mathcal E_2}$.
 \end{proof}

\begin{Rem}\label{symmetry}
The binary operations $\cap_\mathcal O$ on the real phase structures are not commutatives.
If $\mathcal O^T$ is the transposed orientation of $\mathcal O$, then $\mathcal E_1 \cap_{\mathcal O} \mathcal E_2 = \mathcal E_2 \cap_{\mathcal O^T} \mathcal E_1$.
Also notice that the unique 0-real phase structure $\mathcal E_0$ is the left and right identity for the operators $\cap_\mathcal O$.

\end{Rem}

\begin{theo}
    Let $\Gamma$ be triangulation of $\Delta$.
    For any $k$ between 0 and $n$, there exists a $k$-real phase structure on $\Gamma$.
\end{theo}
    
\begin{proof}
    The existence (and uniqueness) of the 0-real phase structure $\mathcal E_0$ is discussed in example \ref{ex0}.

    We choose any sign distribution $\mu$ on the vertices of $\Gamma$, this defines a 1-real phase structure $\mathcal E_\mu$ (see Example \ref{signe}).
    We choose any ordering on the vertices of $\Gamma$, this defines a globally acyclic orientation $\mathcal O$ by orienting each edge from the smallest endpoint to the largest.
    For any $k$ between 1 and $n$, we construct inductively the $k$-real phase structure $\mathcal E_k = \mathcal E_{k-1} \cap_\mathcal O \mathcal E_\mu$.
\end{proof}

Not all real phase structures can be obtained by the operators $\cap_\mathcal O$.
Stronger than that, there exists T-manifolds that are not contained in any other T-manifold except themselves and the full T-manifold of codimension 0.
Figure \ref{fig:counter-ex} shows such a  2-real phase structure $\mathcal E_S$ on a polygon $\Delta_S$.
For each triangle $\tau$ of $\Delta_S$, the set $\mathcal E_S(\tau)$ is a singleton in $\mathcal Q^2 = \{++,+-,-+,--\}$, displayed directly on $\tau$ on the figure.
It is impossible to extend it such that the T-curve contains the 10 points of $X_{\mathcal E_S}$.
Indeed, if one starts by fixing the 1-real phase structure on the edge joining the two interior vertices, the structure $\mathcal E_S$  will impose how to extend the 1-real phase structure on the adjacent edges.
But this local spreading conditions are not consistent globally and the T-curve will always miss at least one of the ten points.

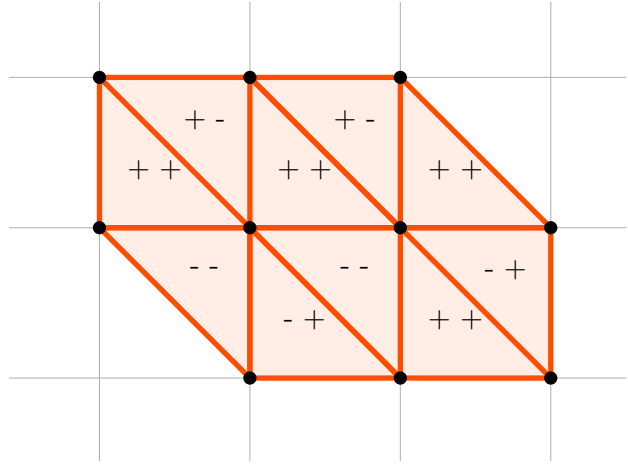
\begin{figure}
    \centering
    \input{polygon}
    \caption{The polygon $\Delta_S$ with the real phase structure on the 2-skeleton $\mathcal E_S$}
    \label{fig:counter-ex}
\end{figure} 

\subsection{Description for T-manifold of codimension 2}

When $k_1=k_2=1$, the real phase structure $\mathcal E_1$ and $\mathcal E_2$ can be encoded by sign distributions $\mu_1$ and $\mu_2$ (see Example \ref{signe}).
In this situation, the description of $\mathcal E = \mathcal E_1 \cap_\mathcal O \mathcal E_2$ becomes especially simple:
it is entirely described by the two sign distributions and the locally acyclic orientation.

As in classical combinatorial patchworking for T-hypersurfaces, all the data can be drawn directly on the triangulation,
and there is a straightforward rule for constructing the T-manifold of codimension 2 from these data.
This makes this description of these T-manifolds of codimension 2 especially easy to present and manipulate.
In particular, it is a convenient way to produce T-curves in 3-dimensional space, and we will use it in Section \ref{sec:maxcurve} to construct maximal curves.

Let $\tau$ be a triangle of $\widetilde{\Gamma}$ with vertices $v_0, v_1$ and $v_2$ ordered by $\mathcal O$.
We call $\tau$ a \emph{compatible triangle} if $\mu_1(v_0) \neq \mu_1(v_1)$ and $\mu_2(v_1) \neq \mu_2(v_2)$.
Equivalently, along the two consecutive edges with the same orientation, the endpoints of the first edge differ for the sign distribution $\mu_1$, and the endpoints of the second differ for $\mu_2$.

Figure \ref{fig:compatibletriangles} presents four acyclic triangles with twice the same sign distribution $\mu_1 = \mu_2$ (represented by the color of the vertices).
Only the top-left triangle is compatible.

\begin{figure}
    \centering
    \includegraphics[width = 5cm]{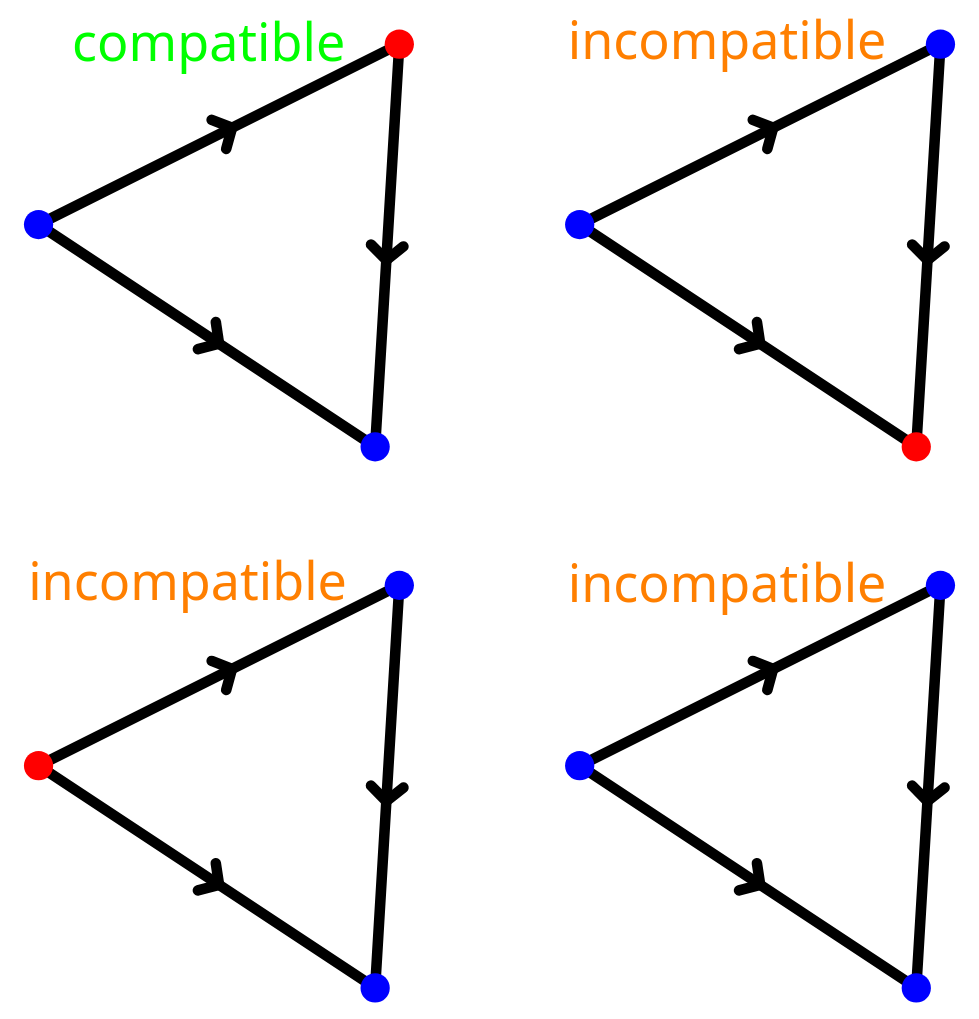}
    \caption{Compatible and incompatible triangles}
    \label{fig:compatibletriangles}
\end{figure}

By Lemma \ref{intersectionrps}, the set $\mathcal E(\tau) = \{s \in \mathcal Q^3 | s(\tau) \text{ is a compatible triangle} \}$ is exactly the 2-real phase structure $\mathcal E_{1} \cap_\mathcal O \mathcal E_{2}$.
Moreover, the T-manifold $X_\mathcal E$ is contained in the intersection of the two T-hypersurfaces define by $\mu_1$ and $\mu_2$.

\begin{Ex}
    Figure \ref{fig:tpoint} illustrates a part of the patchwork in a single triangle.
    The color of the upper half of the vertices of the triangle represents the first distribution of sign associated the orange T-curve.
    The color of the lower half of the vertices of the triangle represents the second distribution of sign associated the green T-curve.
    The T-manifold of codimension 2, which here is only a point, is represented in yellow.
\end{Ex}

\begin{figure}
\includegraphics{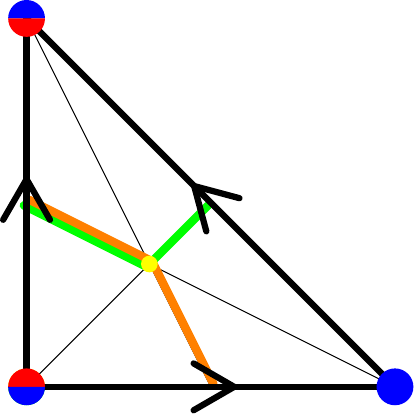}
\caption{Illustration of the real stable intersection}
\label{fig:tpoint}
\end{figure}

In this setting, we have the following realizability theorem for T-manifold of codimension 2.

\begin{theo}\label{realizibility}
    Let $\Delta$ be a smooth lattice polytope, $\Gamma$ be a convex triangulation of $\Delta$, $\mathcal O$ a globally acyclic orientation of $\Gamma$ and two sign distributions $\mu_1, \mu_2$.
    There exists two hypersurfaces $H_1$ and $H_2$ with Newton polytope $\Delta$ in $\mathbb RX_\Delta$ such that,
    for any $i=1,2$, the triple $(\mathbb RX_\Delta, H_i, H_1 \cap H_2)$ is homeomorphic to $(\widetilde{\Delta}, X_{\mathcal E_{\mu_i}}, X_{\mathcal E_{\mu_1} \cap_{\mathcal O} \mathcal E_{\mu_2}})$
    This homeomorphism is stratified in the sense of Theorem \ref{realtoric}.
\end{theo}

The proof of this theorem is developed in the following subsection and is based on Sturmfels' patchworking theorem.

\subsection{Proof of Theorem \ref{realizibility}}

\subsubsection{Sturmfels' patchworking of complete intersection}

There is another distinct approach to define intersection of T-hypersurfaces.
Sturmfels had generalized Viro's theorem to complete intersections of any number of T-hypersurface \cite{sturmfels1994viro}.
Here we extend Sturmfels' construction for any two real phase structures and then restate its theorem for complete intersection of two T-hypersurfaces.

Fix two polytopes $\Delta^{(1)}$ and $\Delta^{(2)}$ with respective subdivisions $\Gamma^{(1)}$ and $\Gamma^{(2)}$.
Consider a subdivision $\Gamma$ of the Minkowski sum $\Delta = \Delta^{(1)}+ \Delta^{(2)}$.
The subdivision $\Gamma$ is called \emph{mixed} if each cell $\gamma$ of $\Gamma$ can be decomposed in the form $\gamma_1 + \gamma_2$, where $\gamma_1 \in \Gamma^{(1)}$ and $\gamma_2 \in \Gamma^{(2)}$,
and this decomposition is compatible with intersection: if $\gamma = \gamma_1 + \gamma_2$ and $\gamma' = \gamma'_1+\gamma'_2$, then $\gamma \cap \gamma' = (\gamma_1 \cap \gamma'_1)+(\gamma_2 \cap \gamma'_2)$.

If, in addition, for each cell, $\dim \gamma = \dim \gamma_1 + \dim \gamma_2$ then the mixed subdivision $\Gamma$ is said to be \emph{fine}.

A mixed subdivision $\Gamma$ of $\Delta=\Delta^{(1)}+\Delta^{(2)}$ is called \emph{coherent} if there are heights functions $\nu_i \colon F_0(\Gamma^{(i)}) \to \mathbb R$ such that $\Gamma^{(1)}$, $\Gamma^{(2)}$ and $\Gamma$ are respectively the convex subdivisions $\Gamma_{\nu_1}$, $\Gamma_{\nu_2}$ and $\Gamma_{\nu_1 \oplus \nu_2}$
where
\[ \nu_1 \oplus \nu_2(x) \colon= \min \{ \nu_1(x_1) + \nu_2(x_2) | x_1 \in \Delta^{(1)}, x_2 \in \Delta^{(2)} \text{ and } x_1+x_2=x\}.\]

Now suppose $\Gamma^{(1)}$ and $\Gamma^{(2)}$ are both (non necessarily convex) triangulations equipped with real phase structures $\mathcal E^{(1)}$ and $\mathcal E^{(2)}$ respectively.
Let $\Gamma$ be a fine mixed subdivision of $\Gamma^{(1)}$ and $\Gamma^{(2)}$.

As in Section \ref{sec:construction}, extend by symmetry the mixed subdivision $\Gamma$ of $\Delta$ into the subdivision $\widetilde{\Gamma}$ of $\widetilde{\Delta}$.
Fix $i=1$ or $2$.
Consider simplices $s(\gamma) \in \widetilde{\Gamma}$ with $\gamma = \gamma_1 + \gamma_2$, where $\gamma_1 \in \Gamma^{(1)}$ and $\gamma_2 \in \Gamma^{(2)}$,
such that $s \in \mathcal E^{(i)}(\gamma_i)$.
We define $Y_{\Gamma, \mathcal E^{(i)}}$ to be the full subcomplex of $\Bary(\widetilde{\Gamma})$ consisting of the barycenters of such simplices $s(\gamma)$.

Sturmfels' result can be restated as follows:

\begin{theo}
    Let $\Delta^{(1)}, \Delta^{(2)}$ be a lattice polytopes, $\Gamma_{\nu_1 \oplus \nu_2}$ a coherent fine mixed subdivision of $\Delta = \Delta^{(1)}+\Delta^{(2)}$ and $\mathcal E^{(1)}, \mathcal E^{(2)}$ real phase structures of codimension 1 respectively on $\Delta^{(1)}, \Delta^{(2)}$.
    There exists algebraic hypersurfaces $H_1, H_2$ with respective Newton polytopes $\Delta^{(1)},\Delta^{(2)}$ in $\mathbb RX_\Delta$ and a homeomorphism $h \colon \widetilde \Delta \to \mathbb RX_\Delta$ such that
    $H_1 \cap H_2$ is a smooth complete intersection and $h(Y_{\Gamma_{\nu_1 \oplus \nu_2}, \mathcal E^{(i)}}) = H_i$, for $i=1,2$.
    The homeomorphism $h$ is stratified in the sense of Theorem \ref{realtoric}.
\end{theo}
For $i=1,2$, let $\mu_i$ be a sign distribution such that $\mathcal E^{(i)}=\mathcal E_{\mu_i}$ (see Example \ref{signe}).
The hypersurfaces $H_i$ are given by the zero sets $\{P_t^{(i)}=0\} \subset \mathbb RX_\Delta$, for sufficiently small $t>0$, of the Viro polynomials

\[ P_t^{(i)}(x) = \sum_{a \in F_0(\Gamma_{\nu_i})} \mu_i(a) t^{\nu_i(a)}x^a.\]

\begin{Ex}
    Figure \ref{fig:Sturmfels} illustrates a part of the Sturmfels' construction.
    For each vertex $u=u_1+u_2$ of the mixed subdivision the upper half represents $\mu_1(u_1)$ while the lower half represents $\mu_2(u_2)$.
    The curve $Y_{\Gamma, \mathcal E_1}$ is represented in orange.
    The curve $Y_{\Gamma, \mathcal E_2}$ is represented in green.
    The intersection $Y_{\Gamma, \mathcal E_1} \cap Y_{\Gamma, \mathcal E_2}$, which here is only a point, is represented in yellow.
\end{Ex}

\begin{figure}
    \includegraphics{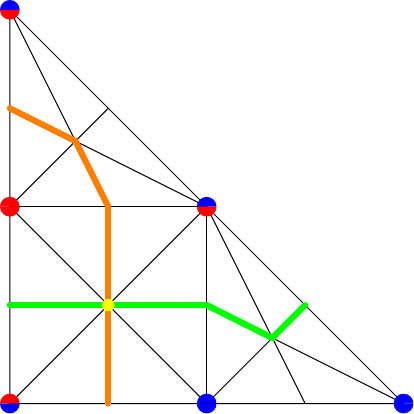}
    \caption{Illustration of Sturmfels's patchworking method}
    \label{fig:Sturmfels}
\end{figure}

\subsubsection{From Sturmfels' method to 2-real phases structures}

A locally acyclic orientation defines a mixed subdivision of $2\Delta= \Delta+ \Delta$ as follows.
Inside $2\Delta$ there is a triangulation $2\Gamma$ made of the simplices $2\tau = \tau + \tau$ with $\tau \in \Gamma$.
For each $n$-simplex $\tau = [v_0, v_1, \dots, v_n]$ ordered by $\mathcal O$,
we subdivide $2\tau$ by the $n+1$ mixed cells $[v_0, v_1, \dots, v_j] +[v_j, v_{j+1}, \dots, v_n]$ for $j=0,1 \dots n$.
All these mixed cells formed the fine mixed subdivision $\Gamma_\mathcal O$ of $2\Delta$.

\begin{prop}\label{regularmixed}
    Let $\Gamma$ be a triangulation.
    The map $\mathcal O \mapsto \Gamma_{\mathcal O}$ defines a bijection between locally acyclic orientation of $\Gamma$ and fine mixed subdivision refining $2\Gamma$.
    Moreover, the mixed subdivision $\Gamma_{\mathcal O}$ is coherent if and only if $\mathcal O$ is globally acyclic and  $\Gamma$ is convex.
\end{prop}
\begin{proof}
    After using the Cayley trick, this is the Lemma 7.2.9 of \cite{de2010triangulations}.
\end{proof}

If we give ourselves two sign distributions $\mu_1, \mu_2 \colon F_0(\Gamma) \to \{-,+\}$, we can create two manifolds of codimension 2:
\begin{itemize}
    \item the intersection $Y_{\Gamma_\mathcal O, \mathcal E_{\mu_1}} \cap Y_{\Gamma_\mathcal O, \mathcal E_{\mu_2}}$ obtained by Sturmfels' method
    \item the T-manifold $X_\mathcal{E}$ associated to the 2-real phase structure $\mathcal E = \mathcal E_{\mu_1} \cap_\mathcal O \mathcal E_{\mu_2}$.
\end{itemize}

These two constructions are strongly related by the following theorem:

\begin{theo}\label{posethomeo}
   Let $\Delta$ be a smooth lattice polytope, $\Gamma$ be a triangulation of $\Delta$, $\mathcal O$ a locally acyclic orientation of $\Gamma$ and two real phase structures $\mathcal E_1, \mathcal E_2$ of codimension 1 on $\Gamma$.
   For any $i=1,2$, the triples of topological spaces $(\widetilde{\Delta}, X_{\mathcal E_i}, X_{\mathcal E_1 \cap_\mathcal O \mathcal E_2})$ and $(\widetilde{2\Delta}, Y_{\Gamma_\mathcal O, \mathcal E_{i}}, Y_{\Gamma_\mathcal O, \mathcal E_{1}} \cap Y_{\Gamma_\mathcal O, \mathcal E_{2}})$ are homeomorphic.
   This homeomorphism is stratified in the sense of Theorem \ref{realtoric}.
\end{theo}

The proof of this theorem will need the following topological result:
\begin{theo}[Weak Schoenflies PL theorem, \cite{Newman, brown}]
    Let $M$ be a PL sphere of dimension $n$.
    For any PL sphere $N$ of dimension $n-1$ in $M$,
    the set $M \setminus N$ is made of two connected components, whose closures are topological $n$-balls.
\end{theo}

\begin{proof}[Proof of Theorem \ref{posethomeo}]
    We fix $\mu_1$ and $\mu_2$ signs distributions associated to $\mathcal E_1$ and $\mathcal E_2$.
    Recall that, by Remark \ref{symmetry}, the role of $\mathcal E_1$ and $\mathcal E_2$ can be interchanged. Thus, we only need to exhibit one of the two homeomorphisms.

    The proof proceeds by constructing regular cell subdivisions $\mathcal C$ of $\widetilde \Delta$ and $\mathcal D$ of $\widetilde{2\Delta}$ adapted to the respective triples and combinatorially isomorphic.
    This will directly yield our result since the combinatoric of regular cell complexes determines them up to cellular homeomorphism.
    
    We begin by defining the regular cell decomposition $\mathcal C$ of $\widetilde \Delta$ with subcomplexes $\mathcal C'$ and $\mathcal{C}''$ that subdivide $X_{\mathcal E_1}$ and $X_{\mathcal E_1 \cap_\mathcal O \mathcal E_2}$, respectively.
    The decomposition $\mathcal C$ refines $\widetilde \Delta$ and is coarser than $\Bary{\widetilde \Delta}$.
    Each simplex $\sigma \in \widetilde \Gamma$ intersecting $X_{\mathcal E_1}$ is divided in two by the T-hypersurface, and $X_{\mathcal E_1}$ is, in turn, divided in two by $X_{\mathcal E_1 \cap_{\mathcal O} \mathcal E_2}$, if they intersect.
    We define

    \[ \mathcal C := \{ \mathcal{C}(\sigma)^+, \mathcal{C}(\sigma)^-, \mathcal{C'}(\sigma)^+, \mathcal C'(\sigma)^-, \mathcal C''(\sigma) \mid \sigma \in \widetilde \Gamma \} \]
    with the following definitions:
    
    \[ \mathcal C(\sigma)^\pm := \bigcup_{\substack{u= \tau_0 < \dots < \tau_n = \sigma \\ \mu_1(u) =\pm}} [\hat \tau_0, \dots, \hat \tau_n], \]

    \[ \mathcal C'(\sigma)^\pm := \bigcup_{\substack{[u,v] = \tau_1 < \dots < \tau_n =\sigma \\ \mu_1(u)  \neq \mu_1(v) \\ \mu_2(u) = \pm}} [\hat \tau_1, \dots, \hat \tau_n]\]

    and

    \[ \mathcal C''(\sigma) = \bigcup_{\substack{[u,v,w] = \tau_2 < \dots < \tau_n = \sigma \\  \mu_1(u)  \neq \mu_1(v) \\ \mu_2(v) \neq \mu_2(w)}} [\hat \tau_2, \dots, \hat \tau_l] = C_{\mathcal E_1 \cap_\mathcal O \mathcal E_2}(\sigma).\]

    Note that some of that cells might be empty.
    First observe that the 
    \[\mathcal C(\sigma)^+ \cup  \mathcal C(\sigma)^- = \bigcup_{\tau_0 < \dots < \tau_n = \sigma} [\hat \tau_0, \dots, \hat \tau_n] = \sigma\]
    and
    \[\mathcal C'(\sigma)^+ \cup  \mathcal C'(\sigma)^- = \bigcup_{\substack{[u,v] =\tau_1 < \dots < \tau_n = \sigma \\ \mu_1(u)  \neq \mu_1(v)}} [\hat \tau_1, \dots, \hat \tau_n] = C_{\mathcal E_1}(\sigma).\]

    Secondly,
     \[\mathcal C(\sigma)^+ \cap  \mathcal C(\sigma)^- = \bigcup_{\substack{[u,v] =\tau_1 < \dots < \tau_n = \sigma \\ \mu_1(u)  \neq \mu_1(v)}} [\hat \tau_1, \dots, \hat \tau_n] = C_{\mathcal E_1}(\sigma)\]

     and 
    \[\mathcal C'(\sigma)^+ \cap  \mathcal C'(\sigma)^- = \bigcup_{\substack{[u,v,w] =\tau_2 < \dots < \tau_n = \sigma \\ \mu_1(v)  \neq \mu_1(w)\\ \mu_2(u) \neq \mu_2(v)}} [\hat \tau_1, \dots, \hat \tau_n] = C_{\mathcal E_1 \cap_\mathcal O \mathcal E_2}(\sigma).\]
    To justify the first equality of the last line, observe that given a triangle $[u,v,w]$ with two edges $[u_1, v_1]$ and $[u_2, v_2]$ such that $\mu_1(u_1) \neq \mu_1(v_1)$, $\mu_1(u_2) \neq \mu_1(v_2)$, $\mu_2(u_1) = +$ and $\mu_2(u_2) = -$, we have either $u_1<u_2$ or $u_2<u_1$. Then either $u=u_1 < v=u_2 < w=v_2$ or $u = u_2 < v=u_1 < w=v_1$ and, in both case, $\mu_2(u) \neq \mu_2(v)$ and $\mu_1(v) \neq \mu_1(w)$.
    
    Let's also define 

    \[ \underline{\mathcal C(\sigma)^\pm} := \bigcup_{\substack{u= \tau_0 < \dots < \tau_{n-1} < \sigma \\ \mu_1(u) =\pm}} [\hat \tau_0, \dots, \hat \tau_{n-1}] = C(\sigma)^\pm \cap \partial \sigma. \]

    Observe that
    
    \[ \underline{\mathcal C(\sigma)^+}  \cap \underline{\mathcal C(\sigma)^-} = \bigcup_{\substack{\tau_1 < \dots < \tau_{n-1} <\sigma \\ \tau_1 = [u,v], ~ \mu_1(u) \neq \mu_1(v)}} [\hat \tau_1, \dots, \hat \tau_{n-1}] = \partial C_{\mathcal E_1}(\sigma).\]
    
    Suppose $C_{\mathcal E_1}(\sigma)$ is non-empty. We apply the weak PL Schoenflies theorem to the pair $(\partial \sigma, \partial C_{\mathcal E_1}(\sigma))$, which are both PL spheres of respective dimension $n-1$ and $n-2$ (the first by definition of PL sphere and the second by Proposition \ref{PLmanifold}). The two closures $\underline{\mathcal C(\sigma)^\pm}$ of the connected components are $(n-1)$-balls.
    Taking the cone over the barycenter of $\sigma$, we recover the cells
    \[ \mathcal C(\sigma)^\pm = \bigcup_{\substack{u= \tau_0 < \dots < \tau_n = \sigma \\ \mu_1(u) =\pm}} [\hat \tau_0, \dots, \hat \tau_n] = \Cone_{\hat \sigma}(\underline{\mathcal C(\sigma)^\pm})  \]
    which thus are $n$-balls.

    Now we similarly subdivide $\partial C_{\mathcal{E}_1}(\sigma)$:

    \[ \underline{\mathcal C'(\sigma)^\pm} := \bigcup_{\substack{[u,v] = \tau_1 < \dots < \tau_{n-1} <\sigma  \\ \mu_1(u)  \neq \mu_1(v) \\ \mu_2(u) = \pm}} [\hat \tau_1, \dots, \hat \tau_{n-1}] = \mathcal C'(\sigma)^\pm \cap \partial \sigma \subset \partial C_{\mathcal E_1}(\sigma).\]

    Observe that
    
    \[ \underline{\mathcal C'(\sigma)^+}  \cap \underline{\mathcal C'(\sigma)^-} = \bigcup_{\substack{[u,v,w]=\tau_2 < \dots < \tau_{n-1} <\sigma \\ \mu_1(v) \neq \mu_1(w) \\ \mu_2(u) \neq \mu_2(v)}} [\hat \tau_1, \dots, \hat \tau_{n-1}] = \partial C_{\mathcal E_1 \cap_\mathcal O \mathcal E_2}(\sigma).\]
    
    Suppose $C_{\mathcal E_1 \cap_\mathcal O \mathcal E_2}(\sigma)$ is non-empty. We apply weak PL Schoenflies theorem again to $(\partial C_{\mathcal E_1}(\sigma), \partial C_{\mathcal E_1 \cap_\mathcal O \mathcal E_2}(\sigma))$.
    Indeed, $\partial C_{\mathcal E_1 \cap_\mathcal O \mathcal E_2}(\sigma)$ and $\partial C_{\mathcal E_1}(\sigma)$ are both PL spheres of respective dimension $n-2$ and $n-3$ by Proposition \ref{PLmanifold}.
    Thus, the closures $\underline{\mathcal C'(\sigma)^\pm}$ of the two components of $\partial C_{\mathcal E_1}(\sigma) \setminus \partial C_{\mathcal E_1 \cap_\mathcal O \mathcal E_2}(\sigma)$ are $(n-2)$-balls.
    Taking the cone over the barycenter of $\sigma$, we recover the cells of $\mathcal C$
     \[ \mathcal C'(\sigma)^\pm = \bigcup_{\substack{u= \tau_0 < \dots < \tau_n = \sigma \\ \mu_1(u) =\pm}} [\hat \tau_0, \dots, \hat \tau_n] = \Cone_{\hat \sigma}( \underline{\mathcal C'(\sigma)^\pm})  \] that are thus $(n-1)$-balls.
     
    Since $\mathcal C''(\sigma) = C_{\mathcal E_1 \cap_\mathcal O \mathcal E_2}(\sigma)$ is an $(n-2)$-ball, all the cells of $\mathcal C$ are balls.

    We verify that the boundary of each cell is a union of cells:
    \[ \partial \mathcal C(\sigma)^\pm = \underline{\mathcal C(\sigma)^\pm} \cup C_{\mathcal E_1}(\sigma) = \bigcup_{\tau < \sigma} \mathcal C(\tau)^\pm \cup \mathcal C'(\sigma)^+ \cup \mathcal C'(\sigma)^-, \]
    
    \[ \partial \mathcal C'(\sigma)^\pm = \underline{\mathcal C'(\sigma)^\pm} \cup C_{\mathcal E_1 \cap_\mathcal O \mathcal E_2}(\sigma) = \bigcup_{\tau < \sigma} \mathcal C'(\tau)^\pm \cup \mathcal C''(\sigma), \]
    and
    \[ \partial \mathcal C''(\sigma) = \bigcup_{\tau < \sigma} \mathcal C''(\tau).\]
    Finally, since the interior of each simplex of the barycentric subdivision  $\Bary(\widetilde \Gamma)$ lies in the interior of a unique cell of $\mathcal C$, any point $x \in \widetilde \Delta$ lies in the interior of one cell of $\mathcal C$.
    It follows that $\mathcal C$ is a regular cell subdivision of $\widetilde \Delta$ with subcomplexes $\mathcal C' = \{ \mathcal C'(\sigma)^+, \mathcal C'(\sigma)^-, \mathcal C''(\sigma) : \sigma \in \widetilde \Gamma \}$ subdividing $X_{\mathcal{E}_1}$ and $\mathcal C'' = \{ \mathcal C''(\sigma) : \sigma \in \widetilde \Gamma \}$ subdividing $X_{\mathcal E_1 \cap_\mathcal O \mathcal E_2}$ (with the canonical cell decomposition).
    
    On the other hand, we have a refinement $\mathcal D$ of the triangulation $\widetilde{2\Gamma}$ of $\widetilde{2\Delta}$ defines using the two T-hypersurfaces $Y_{\Gamma_\mathcal O, \mathcal E_1}$ and $Y_{\Gamma_\mathcal O, \mathcal E_2}$ given by Sturmfels' method.
    In a simplex $\sigma \in \widetilde{\Gamma}$, each T-hypersurface divides $2\sigma \in \widetilde{2\Gamma}$ in two halves:
      \[M_i(\sigma)^\pm := \bigcup_{\substack{u_1+u_2 = \tau_0 < \dots < \tau_{n} \subset 2\sigma  \\ \mu_i(u_i) = \pm}} [\hat \tau_0, \dots, \hat \tau_n]\]
      
     such that $M_i(\sigma)^+ \cap M_i(\sigma)^- = Y_{\Gamma_\mathcal O, \mathcal E_i} \cap \sigma$.

    We define 
    \[ \mathcal D = \{ \mathcal{D}(\sigma)^+, \mathcal{D}(\sigma)^-, \mathcal{D'}(\sigma)^+, \mathcal D'(\sigma)^-, \mathcal D''(\sigma): 2\sigma \in \widetilde{2\Gamma} \} \]  
    with
    \[\mathcal D(\sigma)^\pm := M_1(\sigma)^\pm,\]
    \[\mathcal D'(\sigma)^\pm := M_1(\sigma)^+ \cap M_1(\sigma)^- \cap M_2(\sigma)^\pm \]
    and
    \[ \mathcal D''(\sigma) := M_1(\sigma)^+ \cap M_1(\sigma)^- \cap M_2(\sigma)^+ \cap M_2(\sigma)^-. \]
    
    Here we use Sturmfels' theorem to see that each cell of $\mathcal D$ is topological ball.
    Indeed, for any simplex, the mixed subdivision $\sigma_\mathcal O$ of $2\sigma$ is always coherent and fine.
    Thus, by applying Sturmfels theorem and  considering only the positive orthant $Q^+ = \{[u_0, \dots, u_n], u_i \geq 0 \} \subset \mathbb RP^n$ we have two linear forms $l_1$ and $l_2$ and a homeomorphism $h \colon 2\sigma \to Q^+$ such that
    $h(\mathcal D(\sigma)^\pm) = \{\pm  l_1 \geq 0 \}$, $h(\mathcal D'(\sigma)^\pm) = \{l_1=0\} \cap \{ \pm l_2 \geq 0\}$ and $h(\mathcal D''(\sigma)) = \{l_1 =0\} \cap \{ l_2=0\}$.
    Thus, all the cells are homeomorphic to polyhedron and in particular to topological balls.
    
    By construction $\mathcal{D'} = \{ \mathcal{D}(\sigma)^-, \mathcal{D'}(\sigma)^+, \mathcal D'(\sigma)^-, \mathcal D''(\sigma): \sigma \in \widetilde{2\Gamma} \}$ is a subdivision of $Y_{\Gamma_\mathcal O, \mathcal E_1}$ and $\mathcal{D''} = \{\mathcal D''(\sigma): \sigma \in \widetilde{2\Gamma} \}$ of $Y_{\Gamma_\mathcal O, \mathcal E_1} \cap Y_{\Gamma_\mathcal O, \mathcal E_2}$.
    We also deduce the following boundaries:
     \[ \partial \mathcal D(\sigma)^\pm = \bigcup_{\tau < \sigma} \mathcal D(\tau)^\pm \cup \mathcal D'(\sigma)^+ \cup \mathcal D'(\sigma)^-, \]
    
    \[ \partial \mathcal D'(\sigma)^\pm =  \bigcup_{\tau < \sigma} \mathcal D'(\tau)^\pm \cup \mathcal D''(\sigma), \]
    and
    \[ \partial \mathcal D''(\sigma) = \bigcup_{\tau < \sigma} \mathcal D''(\tau).\]
    
    fhigbonmk 

The map
\[
f : \mathcal C \longrightarrow \mathcal D
\]
defined by
\[
f(\mathcal C(\sigma)^{\pm}) = \mathcal D(\sigma)^{\pm}, \qquad
f(\mathcal C'(\sigma)^{\pm}) = \mathcal D'(\sigma)^{\pm}, \qquad
f(\mathcal C''(\sigma)) = \mathcal D''(\sigma)
\]
is a bijection preserving inclusions.
It induces a homeomorphism $h \colon \widetilde \Delta \to \widetilde{2\Delta}$ such that a cell of $\mathcal C$ is sent on the corresponding cell of $\mathcal D$ (see \cite[Corolary 4.7.9]{OM}). 
In particular, $h$ maps $X_{\mathcal E_1}$ onto $Y_{\Gamma_\mathcal O, \mathcal E_1}$ and $X_{\mathcal E_1 \cap_\mathcal O \mathcal E_2}$ onto $Y_{\Gamma_\mathcal O, \mathcal E_1} \cap Y_{\Gamma_\mathcal O, \mathcal E_2}$.
In addition, since $h(\sigma) = 2\sigma$ for any simplex $\sigma \in \widetilde{\Gamma}$,
the homeomorphism $h$ is stratified.
\end{proof}

\begin{figure}
    \centering
    \begin{subfigure}[t]{0.45\textwidth}

    \includegraphics{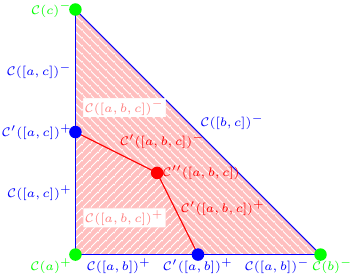}

    \caption{$\mathcal C$}
    \end{subfigure}
    \begin{subfigure}[t]{0.45\textwidth}
        \includegraphics{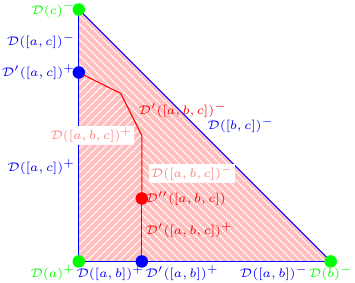}
        \caption{$\mathcal D$}

    \end{subfigure}
    \caption{Illustration of the two regular cell decompositions of the proof on a triangle $[a,b,c]$}
    \label{fig:regsubd}
\end{figure}

\begin{proof}[Proof of Theorem \ref{realizibility}]
    Since $\Gamma_\mathcal O$ is a coherent fine mixed subdivision thanks to Proposition \ref{regularmixed}, 
    Sturmfels' theorem gives the existence of algebraic hypersurfaces $H_1, H_2 \subset \mathbb RX_\Delta$ such that, for any $i=1,2$,
    $(\mathbb RX_\Delta, H_i, H_1 \cap H_2)$ and $(\widetilde{2\Delta}, Y_{\Gamma_\mathcal O, \mathcal E_{i}}, Y_{\Gamma_\mathcal O, \mathcal E_{1}} \cap Y_{\Gamma_\mathcal O,\mathcal E_{2}})$ are homeomorphic.
    By Theorem \ref{posethomeo}, there is a homeomorphism between $(\widetilde{\Delta}, X_{\mathcal E_i}, X_{\mathcal E_1 \cap_\mathcal O \mathcal E_2 })$ and $(\widetilde{2\Delta}, Y_{\Gamma_\mathcal O, \mathcal E_{i}}, Y_{\Gamma_\mathcal O, \mathcal E_{1}} \cap Y_{\Gamma_\mathcal O, \mathcal E_{2}})$.
    The composition of the two homeomorphisms, that both respect the stratification, gives the desired result.
\end{proof}

\subsection{The case of T-curve in dimension 3}

The goal of this subsection is to study the local combinatorics of T-curves obtained by applying the "stable intersection" method in dimension 3.
The results will be used for the analysis of the T-curves in the next section.

Let $(\Delta, \Gamma)$ be a 3-dimensional triangulated polytope, $\mu$ a sign distribution and $ \mathcal O$ a locally acyclic orientation.
We consider the T-curve $X_\mathcal E$ where $\mathcal E = \mathcal E_\mu \cap_\mathcal O \mathcal E_\mu$.

In a tetrahedron $\sigma$ of $\widetilde{\Gamma}$, either $X_\mathcal E$ does not meet $\sigma$ or it intersects two faces of $\sigma$.
In the latter case, we call the common edge of the two faces the axis of  $X_\mathcal E$ in $\sigma$.

For the tetrahedron $\sigma = [v_0,v_1,v_2,v_3]$, with the vertices ordered by $\mathcal O$,
the Table \ref{tab:loc} summarizes the possible axis of $X_\mathcal E$ according to the sign distribution $\mu$ on the vertices of $\sigma$.
When $X_\mathcal E$ does not intersect $\sigma$, the axis is indicated as "none" (empty octant).
It suffices to record the cases where $\mu(v_0) = +$, since inverting all signs does not change the real phase structure.

\begin{table}
    \centering
    \begin{tabular}{cccc|c}
        $v_0$ & $v_1$ & $v_2$ & $v_3$ & axis\\
        \hline
        + & + & + & + & none \\
        + & + & + & - & none \\
        + & + & - & + & $[v_2,v_3]$ \\
        + & + & - & - & none \\
        + & - & + & + & $[v_0,v_1]$ \\
        + & - & + & - & $[v_1,v_2]$ \\
        + & - & - & + & $[v_0,v_3]$ \\
        + & - & - & - & none \\
    \end{tabular}
    \caption{Possible axis of the T-curve $X_\mathcal E$ in a tetrahedron $\sigma$, depending on the sign distribution}
    \label{tab:loc}
\end{table}

We illustrate one case in Figure \ref{fig:localaxis}.
Here we represent the configuration corresponding to the fifth line of the table:
the axis of $X_\mathcal E$ is the edge $[v_0,v_1]$ (in grey).
The sign of a vertex is represented by its color (blue or red).  

\begin{figure}
    \centering
\includegraphics{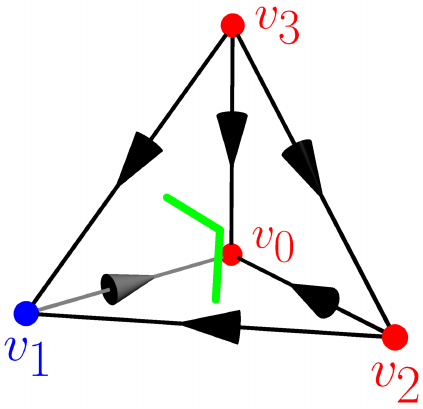}
\caption{Local configuration of the T-curve $X_\mathcal E$ in $[v_0,v_1,v_2,v_3]$ corresponding to the fifth line of table \ref{tab:loc}}
\label{fig:localaxis}
\end{figure}

\begin{prop}\label{propaxis}
    Let $\sigma=[v_0,v_1,v_2,v_3]$ be a tetrahedron in $\Gamma$ with vertices ordered by $\mathcal O$.
    Let $e$ be one of the edges $[v_0,v_1]$, $[v_1,v_2]$, $[v_2,v_3]$ or $[v_3,v_0]$.
    There exists a unique octant $s \in \mathcal Q^3$ such that $s(e)$ is the axis of $X_\mathcal E$ in $s(\sigma)$.
\end{prop}
\begin{proof}
    Since $\sigma$ is unimodular, each of the eight distributions of signs on four vertices (up to inversion of all the signs) appear on exactly one of the eight copies of $\sigma$ (see \cite{Alois} Lemma 1.3.3).
    As shown by Table \ref{tab:loc}, one of these sign distributions is such that the corresponding copy of $e$ is the axis of $X_\mathcal E$ in the corresponding copy of $\sigma$.
\end{proof}

\section{Maximal Curve in codimension 2}\label{sec:maxcurve}

Here we describe, for any positive integer $d$, a pair $(\Sigma_d, C_d)$ of maximal T-surface and maximal T-curve with Newton polytope $d\Delta_3$.
For this, we use the method of Section \ref{stableintersection}.
They will correspond, by Theorem \ref{realizibility}, to a maximal real algebraic surfaces of degree $d$ and maximal real algebraic curve in $\mathbb RP^3$.

In this section, we fix $d \geq 1$ and $\Delta = d \Delta_3$ inside $\mathbb R^3$ with coordinate $(x,y,z)$.

\subsection{Triangulation}

We describe the standard floor triangulation of the tetrahedron $\Delta$.
For each integer $k$ such that $0 \leq k \leq d$, define the slice $T_k = \Delta \cap \{z = d-k\}$.

We give the following triangulation to the slices $T_k$.

For even $k$, we subdivide $T_k$ by the segments $S_k^l = T_k \cap \{y = k-l\}$ for $0 \leq l \leq k$ and the diagonals between the points $(k,2m,d-k)$ and $(0,1+2m,d-k)$ for $0 \leq 2m \leq k$ (see Figure \ref{Tkeven}).

For odd $k$, we subdivide $T_k$ by the segments $S_k^l = T_k \cap \{x = k-l\}$ for $0 \leq l \leq k$ and the diagonals between the points $(2m,k-2m,d-k)$ and $(1+2m,0,d-k)$ for $0 \leq 2m \leq k$ (see Figure \ref{Tkodd}).

\begin{figure}
    \centering
    \begin{subfigure}[t]{0.4\textwidth}

   \includegraphics{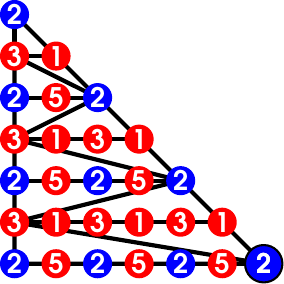}

    \caption{$T_k$ for $k$ even}
    \label{Tkeven}
    \end{subfigure}
    \begin{subfigure}[t]{0.4\textwidth}
        \includegraphics{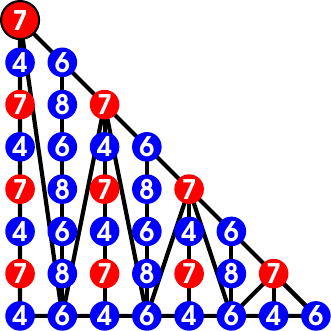}
        \caption{$T_k$ for $k$ odd}
        \label{Tkodd}
    \end{subfigure}
    \caption{The coarse triangulation of $T_k$ and the parities of the vertices}
    \label{fig:Tk}
\end{figure}
The triangulation of each $T_k$ can be uniquely refined in a unimodular triangulation.

Now, we use the triangulation of $T_k$ to define the triangulation of $\Delta$.

For even $k$, take the cone of $T_k$ and its triangulation over $(0,k+1,d-k-1)$, provided that $k+1<d$, and the cone over $(0,k-1,d-k+1)$, provided that $k>0$.

For odd $k$, take the cones of $T_k$ and its triangulation over $(k+1,0,d-k-1)$, provided that $k+1<d$, and the cone over $(k-1,0,d-k+1)$, provided that $k>0$.
The apex of these cones are represented as bigger circles in Figure \ref{fig:Tk}.

Finally, refine the join of the two segments $[(0,0,d-2l),(0,2l,d-2l)]$ and $[(0,0,d-2l - 1), (2l + 1,0,d-2l - 1)]$, provided that $0\leq 2l <d$, and
the join of the two segments $[(0,0,d-2l),(0,2l,d-2l)]$ and $[(0,0,d-2l + 1), (2l - 1,0,d-2l + 1)]$, provided that $0 < 2l \leq d$, in the unique unimodular triangulation.

The result is a unimodular triangulation $\Gamma$ of $\Delta$.

This type of floor triangulation have been used in \cite{Alois} and \cite{itenberg1997topology} for constructing maximal T-surface in dimension 3.
It has also been generalized in \cite{AsympmaxViro} for constructing asymptotically maximal T-hypersurface in any dimension.
In particular, they showed that the triangulation $\Gamma$ is convex.

\begin{figure}
    \centering
    \includegraphics[width=5cm]{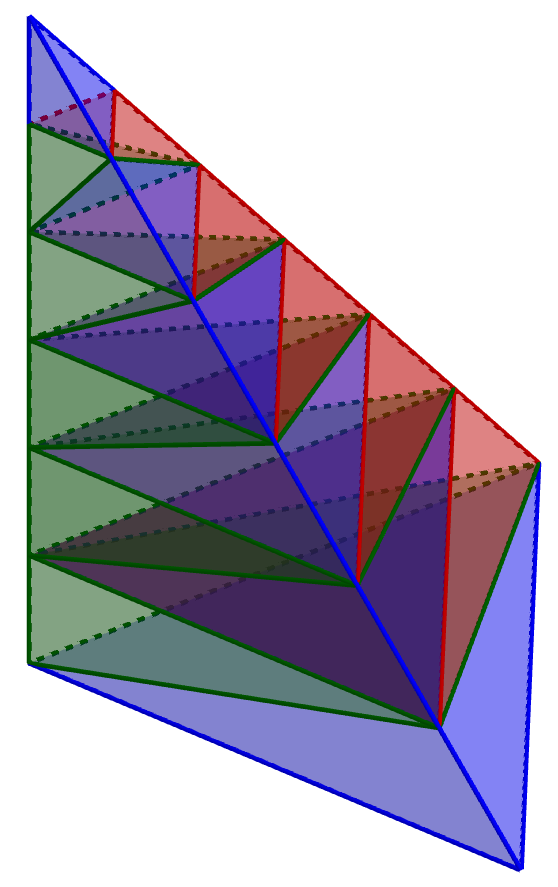}
    \caption{The floor triangulation of $\Delta$}
\end{figure}

\subsection{Parity code, coloring and orientation}
To each integer $v$ vertex of $\Delta$ of coordinate $(x,y,z) \in \mathbb Z^3$ we assign a \emph{parity code} $p(v) \in \{1, 2, 3, 4, 5, 6, 7, 8\}$, depending on the parity of $x,y$ and  $z$, using Table $\ref{tab:parity}$.
These parity codes are represented in the Figure \ref{fig:Tk}.

\begin{table}
    \centering
    \begin{tabular}{ccc|c|c}
        $x$ & $y$ & $d-z$ & parity code & sign\\
        \hline
        odd & odd & even & $1$ & - \\
        even & even & even & $2$ & + \\
        even & odd & even & $3$ & - \\
        even & even & odd & $4$ & + \\
        odd & even & even & $5$ & - \\
        odd & even & odd & $6$ & + \\
        even & odd & odd & $7$ & - \\
        odd & odd & odd & $8$ & + \\
    \end{tabular}
    \caption{Parity code of the vertices}
    \label{tab:parity}
\end{table}

For example $v=(0,0,0)$ has parity code $p(v)=2$ if $d$ is even, and $p(v)=4$ if $d$ is odd.

We give an orientation $\mathcal O$ of the edges $F_1(\Gamma)$ by the following rule. 
Consider an edge of $\Gamma$ with vertices $a$ and $b$. Since $\Gamma$ is unimodular, $a$ and $b$ have distinct parity codes $p(a) = i$ and $p(b) = j$.
We orient the edge from $a$ to $b$ if and only if $i > j$, and from $b$ to $a$ otherwise.
This orientation $\mathcal O$ is globally acyclic because a non-trivial oriented path necessary joins a vertex of parity code $i$ to a vertex of parity code $j$ with $i>j$.

We define the sign distribution $\mu$ as follows.
A vertex $a \in F_0(\Gamma)$ of parity code $p(a) = i$ has a sign $\mu(a) = +1$ if $i$ is even and $\mu(a) =-1$ if $i$ is odd.
See Table \ref{tab:parity}.

The information about the configuration is summarized in Figure \ref{fig:Tk}.
Each vertex is colored according to the sign distribution and numbered by its parity code.

Let us denote $S_d = X_{\mathcal E_\mu \cap_\mathcal O \mathcal E_\mu}$ the T-curve and $\Sigma_d = X_{\mathcal E_\mu}$ the T-surface obtained by patchwork.

\begin{theo}
    There is two real algebraic hypersurfaces $H_1, H_2$ of degree $d$ in $\mathbb RP^3$ such that
    the triples $(\mathbb RP^3, H_i, H_1 \cap H_2)$ is homeomorphic to $(\widetilde{\Delta}, \Sigma_d, C_d)$, for any $i=1,2$.
\end{theo}
\begin{proof}
    Since $\Gamma$ is convex and $\mathcal O$ is globally acyclic, we can apply the Theorem \ref{realizibility}.
\end{proof}

The surface $\Sigma_d$ is a particular case of T-hypersurfaces studied by Itenberg:

\begin{prop}[\cite{itenberg1997topology}]
    The T-surface $\Sigma_d$ is maximal.
    Its connected components are all, except one, homeomorphic to a sphere.
\end{prop}
 
A different proof of the maximality of $\Sigma_d$ is given in \cite[section 2.1]{Alois}.

Our goal is to show the following result:

\begin{theo}\label{maxcurve}
    The T-curve $C_d$ is maximal.
\end{theo}

\subsection{Connected components of $C_d$}

By the formula (\ref{dbound}), the T-curve $C_d$ is maximal if it has $d^3-2d^2+2$ connected components.
All these connected components are cycles homeomorphic to circles.
In the following part, we enumerate the cycles of $C_d$.

\subsubsection{Horizontal cycles}
Let $v_0$ be an interior vertex of $T_k$ and let $v_1 \in F_0(\Gamma)$ be an adjacent vertex of $v_0$ lying outside $T_k$.
Denote $i_0$ and $i_1$ the parities of $v_0$ and $v_1$ respectively.

Consider the star $St(v_0,v_1)$ of the edge $[v_0,v_1]$.
This star contains four other vertices with two distinct parity codes. We denote them $i_2$ and $i_3$, with $i_2 <i_3$.
Observe that there are eight possible quadruplets of parities $(i_0, i_1, i_2, i_3)$ (entirely determined by $i_0$): 
(1,7,2,3), (2,7,3,5), (3,7,1,2), (4,2,6,7), (5,7,2,3), (6,2,7,8), (7,2,4,6) or (8,2,6,7).
For any of these possibilities, we can apply the proposition \ref{propaxis} to find an octant $s$ where the axis of $C_d$ in all the four tetrahedron of $s(St(v_0,v_1))$ is $s([v_0,v_1])
$.
Thus, in the octant $s$, the T-curve $C_d$ contains a cycle around $[s(v_0), s(v_1)]$. See Figure \ref{fig:cyclcearround}.

\begin{figure}
    \centering
    \includegraphics{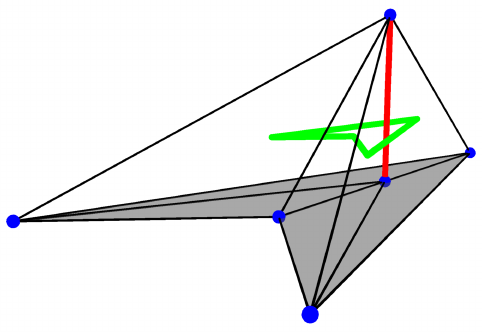}
        \caption{A horizontal cycle}
        The cycle is in green. The grey surface is a portion of a symmetric copy of $T_k$.
        The blue dots are the relevant vertices of $\widetilde \Gamma$.
        The different axis of the cycle are colored in red.
    \label{fig:cyclcearround}
\end{figure}

If $v_0$ lies in the interior of $T_k$, with $2 \leq k \leq d-1$, there are two choices for $v_1$ (one in $T_{k-1}$ and one in $T_{k+1}$).
If $v_0$ lies in the interior of $T_d$, there is only one choice for $v_1$ (in $T_{d-1}$).
Hence, the total number of horizontal cycles is
 \[\sum_{k=2}^{d-1} 2\binom{k-1}{2} + \binom{d-1}{2} = \frac{1}{3}d^3-\frac{3}{2}d^2+\frac{13}{6}d-1 .\]

\subsubsection{Transversal cycles}
Let $v_0$ be a vertex in the interior of $S_k^l$ and let $v_1$ be an adjacent vertex of $v_0$ inside $S_k ^{l \pm 1}$.
Denote $i_0$ and $i_1$ the parities of $v_0$ and $v_1$ respectively.

Consider the star $St(v_0, v_1)$ of the edge $[v_0,v_1]$.
This star contains four other vertices with two distinct parities $i_2$ and $i_3$, with $i_2 <i_3$.
Observe that there are eight possible quadruplets of parities $(i_0, i_1, i_2, i_3)$ (entirely determined by $i_0$): 
(1,2,3,7), (2,3,5,7), (3,2,1,7), (4,6,2,7), (5,3,2,7), (6,7,2,8), (7,6,2,4) or (8,7,2,6).
For any of these possibilities, we can apply the proposition \ref{propaxis} to find an octant $s$ where the axis of $C_d$ in all the four tetrahedron of the $s(St(v_0,v_1))$ is $s([v_0,v_1])
$.
Thus, in the octant $s$, $C_d$ contains a cycle around $[s(v_0), s(v_1)]$. See Figure \ref{fig:transversalcycle}.

If $v_0$ lies in the interior of $S_k^l$, with $1 \leq k \leq d-1$ and $1 \leq l \leq k-1$, there are two choices for $v_1$ (one in $S_k^{l-1}$ and one in $S_k^{l+1}$).
If $v_0$ lies in the interior of $S_k^k$, with $1 \leq k \leq d-1$, there is only one choice for $v_1$ (in $S_k^{k-1}$).
Hence, the total number of transversal cycle is 
\[\sum_{k=1}^{d-1} \sum_{l=1}^{k-1}2(l-1)+\sum_{k=1}^{d-1}(k-1) = \frac{1}{3}d^3-\frac{3}{2}d^2+\frac{13}{6}d-1.\]

\begin{figure}
    \centering
    \includegraphics{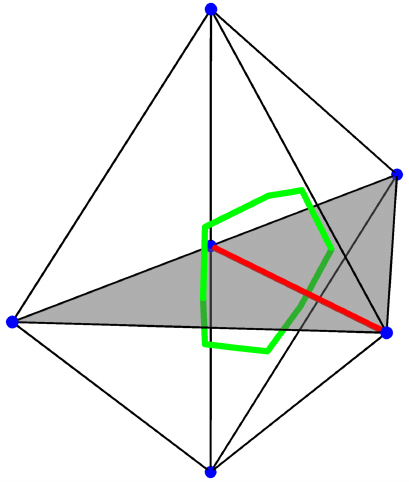}
        \caption{A transversal cycle}
        \label{fig:transversalcycle}
\end{figure}
   
\subsubsection{Pure join cycles}
Let $v_0$ be an interior vertex of $S_k^k$ (for $0 \leq k\leq d-1$) and let $v_1$ be an interior vertex of $S^{k+1}_{k+1}$.
These two vertices form an edge of $\Gamma$.
Denote $i_0$ and $i_1$ the parities of $v_0$ and $v_1$ respectively.

Consider the star $St(v_0,v_1)$ of the edge $[v_0,v_1]$.
This star contains four other vertices and four tetrahedra.
For every tetrahedron of $St(v_0,v_1)$, the set of parities of the vertices is always $\{2,4,5,7\}$
and the pair $\{i_0, i_1\}$ is one of the four possibilities: $\{2,4\}, \{2,7\}, \{5,4\}$ or $\{5,7\}$.
For any of these possibilities, we can apply the proposition \ref{propaxis} to find an octant $s$ where the axis of $C_d$ in all the four tetrahedron of $s(St(v_0,v_1))$ is $s([v_0,v_1])$.
Thus, in the octant $s$, $C_d$ contains a cycle around $[s(v_0),s(v_1)]$. See Figure \ref{fig:joincycle}.

Since $S_k^k$ contains $k-1$ interior vertices and $S^{k+1}_{k+1}$ contains $k$ interior vertices,
the total number of pure join cycles is
\[\sum_{k=0}^{d-1} k(k-1) = \frac{1}{3}d^3-d^2+\frac{2}{3}d.\]

\begin{figure}
    \includegraphics{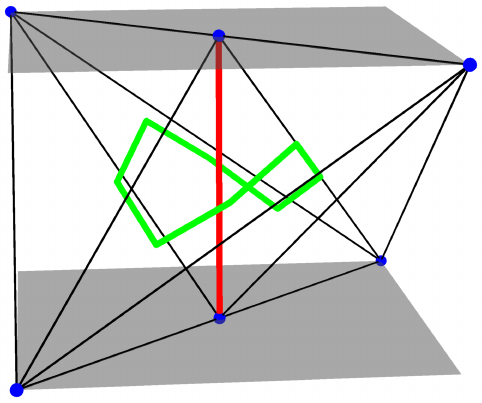}
    \caption{A pure join cycle}
    \label{fig:joincycle}
\end{figure}

\subsubsection{Boundary join cycles}
These cycles are similar in nature to pure join cycles, but one of their vertices lies at the cone point of the triangulation rather than strictly inside the join.

Let $v_0$ be an interior vertex of $S_k^k$ and $v_1$ be the extremity of $S^{k \pm 1}_{k \pm 1}$ on which the cone was taken in the triangulation $\Gamma$.
Denote by $i_0$ and $i_1$ the parities of $v_0$ and $v_1$, respectively.

Consider the star $St(v_0,v_1)$ of the edge $[v_0,v_1]$. This star contains four other vertices:
\begin{itemize}
    \item $v_2$ and $v_3$ inside $S_k^k$ both of parity code $i_2$
    \item $v_4$ in $S^{k \pm 1}_{k \pm 1}$ of parity code $i_3$ 
    \item $v_5$ an extremity of $S_k^{k-1}$ of parity code $i_4$.
\end{itemize}
The star $St(v_0,v_1)$ is the union of four tetrahedra: $[v_0,v_1,v_2,v_4]$, $[v_0,v_1,v_2,v_5]$, $[v_0,v_1,v_3,v_4]$ and $[v_0,v_1,v_3,v_5]$.

There are four possible quintuplets of parities $(i_0, i_1, i_2, i_3, i_4)$ (entirely determined by $i_0$): 
(2,7,5,4,3), (5,7,2,4,3), (4,2,7,5,6) or (7,2,4,5,6).

In each case there exists an octant $s$ where the axis of $C_d$ in all the four tetrahedra of $s(St(v_0,v_1))$ is $s([v_0,v_1])$.
These octants are, respectively $s=(+,+,-)$, $(+,-,+)$, $(-,+,-)$ and $(+,+,-)$.
Thus, there is always an octant $s$ such that the curve $C_d$ form a cycle around the edge $[s(v_0),s(v_1)]$. See Figure \ref{fig:bjcycle}

\begin{center}
    \includegraphics{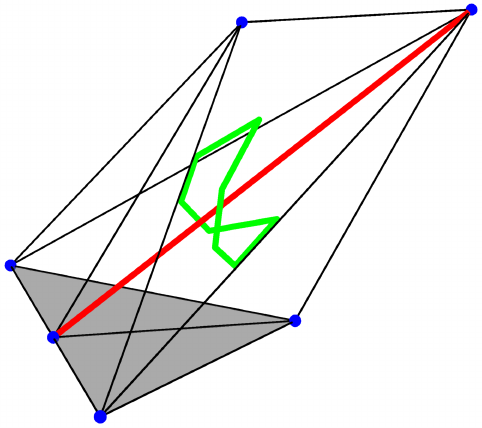}
    \captionof{figure}{A boundary join cycle}
    \label{fig:bjcycle}
\end{center}

\begin{figure}[b]

\end{figure}

If $v_0$ lies in the interior of $S_k^k$, with $2 \leq k \leq d-1$, there are two choices for $v_1$ (one in $S_{k+1}^{k+1}$ and one in $S_{k-1}^{k-1}$).
If $v_0$ lies in the interior $S_d^d$ there is only one choice for $v_1$ (in $S_{d-1}^{d-1}$).
Hence, the total number of such cycles is 
\[\sum_{k=2}^{d-1}2(k-1)+ (d-1) = d^2-2d+1.\]

\subsubsection{Axisless cycles}
Let $v_0$ be an extremity of $S_k^l$ with $1\leq l \leq k -1$. Denote $F$ the facet of $\Delta$ containing $v_0$ and $i_0$ the parity of $v_0$.
We distinguish three cases:

\subsubsection*{Case A} $v_0$ has exactly three neighbors inside $T_k$.

In $\widetilde \Gamma$, the star of $v_0$ consists of eight tetrahedra and six other vertices:
\begin{itemize}
    \item the two vertices of $S_k^{l \pm 1}$ adjacent to $v_0$. Denote by $i_1$ their parity code.
    \item the two vertices of $T_{k \pm 1}$  adjacent to $v_0$. Denote by $i_2$ their parity code.
    \item the vertex of $S_k^k$ adjacent to $v_0$ and its copy along the facet $F$ of $\Delta$.
            Denote by $i_3$ their parity code. 
\end{itemize}
There are four possible quadruplets of parities $(i_0, i_1, i_2, i_3)$: $(1, 2, 7, 3)$, $(2, 3, 7, 5)$, $(4, 6, 2, 7)$ or $(6, 7, 2, 8)$.
In each case there is a pair of octants $\{s, s'=s+e_F\}$ where a cycle of $C_d$ is contained in the star $St(s(v_0))$ of $s(v_0)$ in $\widetilde{\Gamma}$. See Figure \ref{Acycle}.
Observe that since $v_0 \in F$, $s(v_0) = s'(v_0)$.
The pair of octants is respectively $\{(+,+,-), (-,-,+)\}$, $\{(+,+,+), (-,+,+)\}$, $\{(-,+,-), (-,-,-)\}$ or $\{(+,+,-), (-,-,+)\}$.

\begin{center}
    \includegraphics{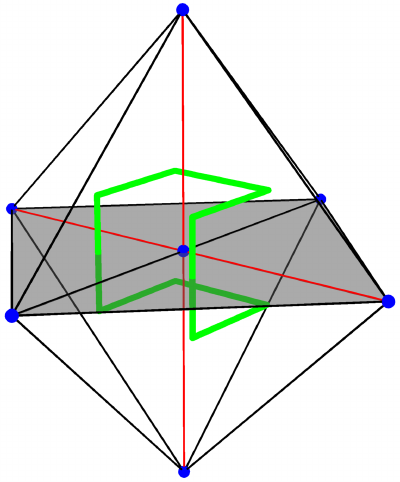}
    \captionof{figure}{An axisless cycle of type A}
    \label{Acycle}
\end{center}

\subsubsection*{Case B} $v_0$ has at least four neighbors inside $T_k$ and $l \neq 1$.

Let:
\begin{itemize}
    \item $v_1$ the vertex of $S_k^l$ adjacent to $v_0$, with parity code $i_1$
    \item $v_2$ (resp. $v_5$) the vertex of  $S^{l-1}_k$  (resp. $S^{l+1}_k$) adjacent to $v_0$ and $v_1$, with parity code $i_2$
    \item $v_3$ (resp. $v_6$) the vertex of $S_k^{l-1}$ (resp. $S^{l+1}_k$) adjacent to $v_2$ (resp. $v_4$), with parity code $i_3$
    \item $v_4$ (resp. $v_7$) the vertex of $T_{k-1}$ (resp. $T_{k+1}$) adjacent to $v_0$ with parity code $i_4$.
\end{itemize}
There are four possible quintuplets of parities $(i_0, i_1, i_2, i_3, i_4)$: $(2, 5, 3, 1, 7)$,
$(3, 1, 2, 5, 7)$, $(6, 8, 7, 4, 2)$ or $(7, 4, 6, 8, 2)$.
In each case there is an octant $s \in \mathcal Q^3$ where a cycle of $C_d$ in the full subcomplex of $\widetilde{\Gamma}$ defined by the vertices $s(v_i)
$, $0 \leq i \leq 7$. See Figure \ref{Bcycle}.
These octants are respectively $s=(+,+,-)$, $(-,+,-)$, $(-,-,+)$ and $(+,+,-)$.

\begin{center}

\includegraphics{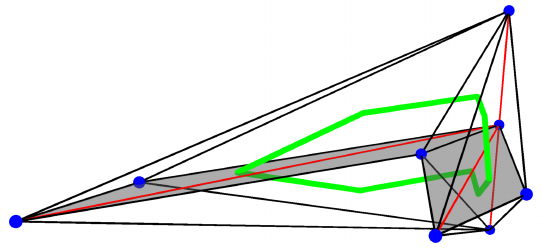}
    \captionof{figure}{An axisless cycle of type B}
    \label{Bcycle}
\end{center}

\subsubsection*{Case C} The remaining case is when $v_0$ is the extremity of $S_k^1$, having five neighbors in $T_k$.
The parity code of $v_0$ is either $3$ or $7$.
Let:
\begin{itemize}
    \item $v_1$ be the unique vertex of $S^0_k$,
    \item $v_2$ be the other extremity of $S^1_k$,
    \item $v_3$ and $v_4$ be the two vertices of $S_k^2$ not on the same edge of $T_k$ than $v_0$ and $v_1$,
    \item $v_5$ (resp. $v_6$) be the vertex of $T_{k-1}$ (resp. $T_{k+1}$) neighbor of the $v_i$, $0 \leq i \leq 4$.
\end{itemize}
Denote by $F$  the face of $\Delta$ containing $v_0, v_1, v_5$ and $v_6$. 
For an octant $s \in \mathcal Q^3$,  let $\mathcal H_s$ be the full subcomplex of $\widetilde{\Gamma}$ of the seven vertices
$s(v_0)=s'(v_0), s(v_1)=s'(v_1), s(v_2), s'(v_2), s(v_3), s(v_4), s(v_5)=s'(v_5)$ and $s(v_6)=s'(v_6)$.
There exists an orthant $s$ such that $C_d$ contains a unique cycle passing through every tetrahedron in $\mathcal H_s$.
See Figure \ref{Ccycle}.
The couple of octants $(s,s')$ is $((-,+,-), (+,+,-))$ if $p(v_0)=3$ and $((+,+,-), (+,+,+))$ if $p(v_0)=7$.

\begin{figure}[h]
    \includegraphics{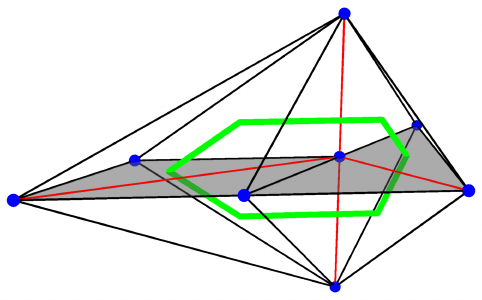}
        \caption{An axisless cycle of type C}
        \label{Ccycle}
\end{figure}

Combining cases A, B and C, we have described a cycle for each extremity of $S^l_k$ with $1 \leq l \leq k-1$ and $2 \leq k \leq d-1$.
Thus, the total number of such cycles is
\[\sum_{k=2}^{d-1} \sum_{l=1}^{k-1}2 = (d-1)(d-2)=d^2-3d+2.\]

\subsubsection{The last cycle}
There is one additional cycle, not localized in a small number of simplices.

This cycle can be detected by observing that, none of the previously described cycles pass through a copy of $T_0$.
But all interior faces must have exactly two copies intersecting $C_d$, this implies the existence of  an additional cycle.

\subsubsection{Conclusion}
If we add up all the cycles we have described we obtain a total of \[d^3 -2d^2+2\] distinct cycles.
So $b_0(C_d) \geq d^3 -2d^2+2$. Thus, $C_d$ is maximal.
This completes the proof of the Theorem \ref{maxcurve}.

\subsection{Concluding remarks}

\begin{Rem}
    By construction, we have the inclusion $C_d \subset \Sigma_d$.
    Being careful, it is possible to describe the topology of the embedding of $C_d$ in $\Sigma_d$.
    For instance, consider the star of an interior vertex of parity code $3$ in the orthant $(-,-,-)$.
    In this star, the T-surface $\Sigma_d$ has one spherical component enclosing the vertex of parity code 3,
    while the T-curve $C_d$ has exactly two horizontal cycles lying in this sphere. The situation is represented in Figure \ref{fig:relat}.
    This shows that the connected components of $C_d$ are distributed between different components of $\Sigma_d$, and do not all belong to a single one.
    
    In particular, consider the real algebraic surface $\mathcal S_d$ and curve $\mathcal C_d$ corresponding respectively to $\Sigma_d$ and $C_d$, by Theorem \ref{realizibility}.
    Then, the spreading of the connected components of $\mathbb R \mathcal C_d$ in multiple components of $\mathbb R \mathcal C_d$ implies that the pair $(\mathcal S_d, \mathcal C_d)$ is not maximal in the sense of Smith-Thom inequality:
    \[ \sum_i \dim H_i(\mathbb R \mathcal S_d, \mathbb R \mathcal C_d; \mathbb F_2) < \sum_i \dim H_i(\mathbb C \mathcal S_d, \mathbb C \mathcal C_d; \mathbb F_2) \]
\end{Rem}

\begin{figure}[h]
\centering
    \includegraphics{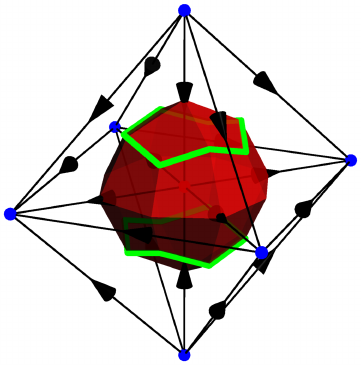}
\caption{A component of the T-surface (in red) containing two cycles of the T-curve (in green)}
\label{fig:relat}
\end{figure}

\begin{Rem}
    The triangulation $\Gamma$  possesses the special property that every tetrahedron has, at least, an edge on the boundary of $\Delta$.
    This condition is crucial: by an argument similar to the proof of Theorem \ref{Kn}, any tetrahedron lacking such a boundary edge yields a subdivision of $K_5$ in the dual graph $G_\Gamma$
    and consequently prevents the existence of a maximal T-curve.
    We think families of triangulations of dilations of $3$-polytopes sharing this property are rare.
    
    For example, consider $\Delta$ the pyramid over a Nakajima polygon defined as the convex hull of the vertices $(0,0,0)$, $(0,\mu,0)$, $(\lambda, \mu,0)$, $(\lambda +\alpha \mu,0,0)$ and $(0,0,1)$,
    where $\alpha, \lambda, \mu$ natural numbers such that $\lambda, \mu \neq 0$.
    Bertrand (\cite[chapter 8]{BertrandThese}) constructed maximal real algebraic curves in $\mathbb RX_\Delta$ using Sturmfels' patchworking method (rather than via real phase structure).
    
    The dilation $d\Delta$ of $\Delta$ can be triangulated in the same fashion as for $d\Delta_3$ by slicing it horizontally.
    See Figure \ref{fig:Nakajima}.
    However, for any choice of apex for the cone over a slice, one inevitably creates tetrahedra without any boundary edge (in green on Figure \ref{fig:Nakajima}).
    
    Thus, our construction cannot be directly extended to $\Delta$ to obtain maximal T-curve.
\end{Rem}

\begin{figure}[h]
    
\centering
\includegraphics{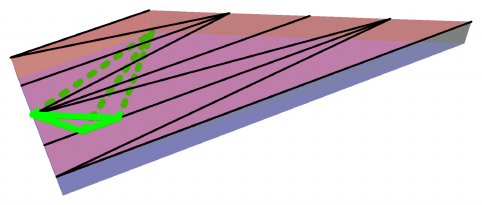}
\caption{The triangulation between two slices of Nakajima polygons, with in green a tetrahedron with no boundary edge in the whole pyramid}
\label{fig:Nakajima}
\end{figure}

\bibliographystyle{alpha}
\bibliography{references} 

\end{document}

%% file: polygon.tex
\pgfplotsset{compat=1.15}
\usetikzlibrary{arrows}
\pagestyle{empty}
\tikzstyle{every node}=[font=\large]
\definecolor{ffvvqq}{rgb}{1,0.3,0}
\definecolor{cqcqcq}{rgb}{0.7,0.7,0.7}
\begin{tikzpicture}[line cap=round,line join=round,>=triangle 45,x=1cm,y=1cm]
\draw [color=cqcqcq,, xstep=2cm,ystep=2cm] (4.8,-5.1) grid (13,1);
\fill[line width=2pt,color=ffvvqq,fill=ffvvqq,fill opacity=0.1] (6,-2) -- (6,0) -- (8,-2) -- cycle;
\fill[line width=2pt,color=ffvvqq,fill=ffvvqq,fill opacity=0.1] (6,-2) -- (8,-4) -- (8,-2) -- cycle;
\fill[line width=2pt,color=ffvvqq,fill=ffvvqq,fill opacity=0.1] (8,-4) -- (10,-4) -- (8,-2) -- cycle;
\fill[line width=2pt,color=ffvvqq,fill=ffvvqq,fill opacity=0.1] (10,-4) -- (10,-2) -- (8,-2) -- cycle;
\fill[line width=2pt,color=ffvvqq,fill=ffvvqq,fill opacity=0.1] (10,-4) -- (12,-4) -- (10,-2) -- cycle;
\fill[line width=2pt,color=ffvvqq,fill=ffvvqq,fill opacity=0.1] (12,-4) -- (10,-2) -- (12,-2) -- cycle;
\fill[line width=2pt,color=ffvvqq,fill=ffvvqq,fill opacity=0.1] (12,-2) -- (10,0) -- (10,-2) -- cycle;
\fill[line width=2pt,color=ffvvqq,fill=ffvvqq,fill opacity=0.1] (10,0) -- (8,0) -- (10,-2) -- cycle;
\fill[line width=2pt,color=ffvvqq,fill=ffvvqq,fill opacity=0.1] (8,0) -- (8,-2) -- (10,-2) -- cycle;
\fill[line width=2pt,color=ffvvqq,fill=ffvvqq,fill opacity=0.1] (8,0) -- (6,0) -- (8,-2) -- cycle;
\draw [line width=2pt,color=ffvvqq] (6,-2)-- (6,0);
\draw [line width=2pt,color=ffvvqq] (8,-2)-- (6,-2);
\draw [line width=2pt,color=ffvvqq] (6,-2)-- (8,-4);
\draw [line width=2pt,color=ffvvqq] (8,-4)-- (8,-2);
\draw [line width=2pt,color=ffvvqq] (8,-2)-- (6,-2);
\draw [line width=2pt,color=ffvvqq] (8,-4)-- (10,-4);
\draw [line width=2pt,color=ffvvqq] (10,-4)-- (8,-2);
\draw [line width=2pt,color=ffvvqq] (8,-2)-- (8,-4);
\draw [line width=2pt,color=ffvvqq] (10,-4)-- (10,-2);
\draw [line width=2pt,color=ffvvqq] (10,-2)-- (8,-2);
\draw [line width=2pt,color=ffvvqq] (8,-2)-- (10,-4);
\draw [line width=2pt,color=ffvvqq] (10,-4)-- (12,-4);
\draw [line width=2pt,color=ffvvqq] (12,-4)-- (10,-2);
\draw [line width=2pt,color=ffvvqq] (10,-2)-- (12,-2);
\draw [line width=2pt,color=ffvvqq] (12,-2)-- (12,-4);
\draw [line width=2pt,color=ffvvqq] (12,-2)-- (10,0);
\draw [line width=2pt,color=ffvvqq] (10,0)-- (10,-2);
\draw [line width=2pt,color=ffvvqq] (10,-2)-- (12,-2);
\draw [line width=2pt,color=ffvvqq] (10,0)-- (8,0);
\draw [line width=2pt,color=ffvvqq] (8,0)-- (10,-2);
\draw [line width=2pt,color=ffvvqq] (8,0)-- (8,-2);
\draw [line width=2pt,color=ffvvqq] (8,-2)-- (10,-2);
\draw [line width=2pt,color=ffvvqq] (10,-2)-- (8,0);
\draw [line width=2pt,color=ffvvqq] (8,0)-- (6,0);
\draw [line width=2pt,color=ffvvqq] (6,0)-- (8,-2);
\begin{scriptsize}
\fill(6,-2) circle (2.5pt);
\fill(6,0) circle (2.5pt);
\fill(8,-2) circle (2.5pt);
\draw (6.734205399189429,-1.2254614515649909) node {+ +};
\fill(8,-4) circle (2.5pt);
\draw (7.399636530249953,-2.5650793864631485) node {- -};
\fill(10,-4) circle (2.5pt);
\draw (8.730498792371002,-3.2305105175236712) node {- +};
\fill(10,-2) circle (2.5pt);
\draw (9.395929923431527,-2.5650793864631485) node {- -};
\fill(12,-4) circle (2.5pt);
\draw (10.735547858329689,-3.2305105175236712) node {+ +};
\fill(12,-2) circle (2.5pt);
\draw (11.400978989390213,-2.5650793864631485) node {- +};
\fill(10,0) circle (2.5pt);
\draw (10.735547858329689,-1.2254614515649909) node {+ +};
\fill(8,0) circle (2.5pt);
\draw (9.395929923431527,-0.5600303205044682) node {+ -};
\draw (8.730498792371002,-1.2254614515649909) node {+ +};
\draw (7.399636530249953,-0.5600303205044682) node {+ -};
\end{scriptsize}
\end{tikzpicture}